\documentclass[12pt]{article}

\usepackage{amssymb, amsmath, amsfonts}
\usepackage{ulem}
\usepackage{graphicx}
\usepackage{setspace} 
\usepackage{rotating}
\usepackage[utf8]{inputenc}
\usepackage[authoryear]{natbib}
\usepackage{array, epsfig, fancyheadings, rotating}
\usepackage{caption}
\usepackage{caption}
\usepackage{lineno}
\usepackage{subcaption}
\usepackage{hyperref}
\usepackage{mathrsfs}
\usepackage{titling}
\usepackage{booktabs}
\usepackage{multirow}
\usepackage{threeparttable}
\usepackage{adjustbox}
\usepackage{xcolor}
\usepackage{amsthm}
\usepackage{lscape}
\usepackage{afterpage}
\usepackage{enumitem}
\usepackage{algorithm}
\usepackage{algpseudocode}
\usepackage{authblk}
\usepackage{geometry}

\hypersetup{
  colorlinks=true,
  linkcolor=blue,
  urlcolor=blue,
  citecolor=blue,
}

\theoremstyle{plain}
\newtheorem{theorem}{Theorem}[section]
\newtheorem{lemma}[theorem]{Lemma}
\newtheorem{corollary}[theorem]{Corollary}
\newtheorem{proposition}[theorem]{Proposition}

\theoremstyle{definition}

\newtheorem{example}{Example}
\newtheorem{remark}{Remark}

\newtheorem{notation}{Notation}

\newcommand{\sB}{\mathcal{B}}

\newcommand{\sX}{\mathcal{X}}

\newcommand{\bbP}{\mathbb{P}}

\newcommand{\bbR}{\mathbb{R}}

\newcommand{\E}{\mathbb{E}}

\title{\vspace{-2em}
  \Large Nonparametric Regression with Measurement Error in Banach Spaces
  \vspace{-0.6em}
}

\author[1]{Pratim Guha Niyogi%
  \thanks{\texttt{pguhaniyogi@umc.edu}}}
\author[2]{Priyadarshi Dey\thanks{\texttt{deyp@millsaps.edu, priyadarshid4@gmail.com}}}%

\affil[1]{Department of Data Science, University of Mississippi Medical Center, Jackson, MS}
\affil[2]{Department of Mathematics, Millsaps College, Jackson, MS}

\date{\vspace{-1.5em}}

\begin{document}

\maketitle

\vspace{-1.0cm}

\begin{abstract}

We consider nonparametric regression with a Banach-space-valued predictor where the covariate takes values in a general separable Banach space. Our primary objective is to develop a unified framework for estimating the regression operator without relying on an inner-product structure, thereby extending classical and functional nonparametric regression beyond Euclidean and Hilbert spaces in the presence of measurement error. We establish the fundamental properties of the proposed estimator and develop its large-sample theory under conditions formulated in terms of the geometry and local concentration of the Banach-valued predictor. In particular, we study pointwise estimation and inference and discuss the construction of pointwise confidence intervals and uniform confidence bands. Our proposed framework is further extended to accommodate error-in-variables settings, where the predictor is observed with contamination. We also investigate related inverse problems arising from the construction of inverse weighting functions and consider partial contamination models. These developments provide a general framework for nonparametric regression and inference with complex infinite-dimensional predictors and identify several statistical and mathematical challenges that arise when the usual Hilbert-space structure is unavailable.

\end{abstract}

\noindent\textbf{Keywords:} pointwise confidence intervals; uniform
confidence bands; inverse weighting; Banach space regression

\section{Introduction}
Regression analysis provides a fundamental framework for describing and predicting the relationship between an outcome and one or more explanatory variables. In its classical nonparametric formulation, suppose that a real-valued response $Y \in \bbR$ and a predictor $X \in \bbR$ satisfy $Y=m(X)+\varepsilon$ with $\mathbb{E}\{\varepsilon\mid X\}=0$ where $m(x)=\mathbb{E}\{Y\mid X=x\}$ is the unknown regression function of primary interest. Since no particular parametric form is imposed on $m$, its value at a given point $x$ can be estimated using local smoothing techniques that borrow information from observations whose predictor values lie close to $x$. Such estimation is based on a sample of paired observations $\{\left(X_i,Y_i\right): i =1,\cdots,n\}$, where $n$ denotes the sample size. For example, a widely used example is the Nadaraya–Watson estimator, which employs a kernel function $K: \bbR \rightarrow [0, \infty)$ satisfying $\int K(u) du = 1$ and $K(u) = K(-u)$ for $u\in \bbR$ to assign observation-specific weights of the form $K((X_{i} - x)/h)$, where $h > 0$ is a bandwidth parameter. The bandwidth determines the size of the local neighborhood and consequently controls the degree of smoothing.  
\par
In many applications, the latent predictor $X_i$ is unavailable, and only an error-contaminated version $Z_i=X_i+U_i$ is observed, where $U_i$ denotes an unobserved measurement error. A seemingly natural strategy is to replace $X_i$ by $Z_i$ in the standard kernel-based Nadaraya–Watson estimator. This substitution fundamentally changes the target of the smoothing procedure. The neighborhood determined by $Z_i$ does not generally coincide with the neighborhood determined by the latent $X_i$, and the resulting estimator typically targets $\mathbb{E}\{Y\mid Z=x\}$, rather than the desired regression function $\mathbb{E}\{Y\mid X=x\}$. Consequently, even as the sample size increases, the naive estimator may remain biased and fail to consistently estimate $m(x)$. Recovering the regression relationship associated with the latent predictor, therefore, requires an explicit correction for the measurement-error mechanism, commonly through deconvolution, corrected-score, or related errors-in-variables techniques \citep{fan1993nonparametric, carroll2006measurement, delaigle2014nonparametric}.
\par
Rapid advances in data collection technologies have made it increasingly common to observe complex data as functions evolving over time, space, or both. This development has stimulated substantial interest in statistical methods for functional data analysis \citep{ramsay2005functional, ferraty2006nonparametric, horvath2012inference, wang2016functional, hsing2015theoretical, kokoszka2017introduction, crainiceanu2024functional}. Much of the existing literature has been formulated within Hilbert spaces, where the availability of inner-product geometry has facilitated the development of a mature and comprehensive theoretical framework. In many applications, functional observations are naturally represented by continuous, and often smoother, sample paths. It is therefore appealing to formulate their analysis directly in the Banach space of continuous functions rather than embedding them in a Hilbert space solely for analytical convenience. Such a formulation preserves the uniform geometry of the observed trajectories but also introduces substantial theoretical challenges because several tools associated with inner-product spaces are no longer available. These difficulties become particularly pronounced when the functional predictor is observed with measurement error. Motivated by this problem, the primary objective of this paper is to investigate how measurement error affects classical nonparametric regression when the predictor takes values in a Banach space and to develop an appropriate estimation framework that accounts for such contamination. In Section \ref{sec:method}, we formulate the problem and introduce a new statistical framework for analyzing Banach space-valued random variables. In Section \ref{sec:results}, we present the basic theoretical results underlying the proposed methodology and systematically state the assumptions required for the subsequent developments. We also discuss the role and interpretation of these assumptions, providing a common theoretical foundation that will be repeatedly used in the later sections. 
\par
Having established the general framework and the basic theoretical results, we consider several classical problems in nonparametric regression in Section \ref{sec:example}. In particular, we examine how these problems can be formulated and studied when the predictor is a Banach space-valued random element. We begin with pointwise inference in Section \ref{sec:cb-point} for the regression function $m(x)$, including pointwise consistency, convergence rates, and asymptotic normality of the proposed estimator, which in turn provide the theoretical basis for pointwise hypothesis testing and the construction of pointwise confidence intervals. In Section \ref{sec:cb-uniform}, we further establish uniform convergence of the proposed estimator over suitable subsets of the predictor space and use the resulting uniform asymptotic theory to construct simultaneous confidence bands for $m$. 

An important feature of the error-in-variables setting is its connection with statistical inverse problems. When the predictor is observed with measurement error, the data provide only a distorted version of the latent predictor, and estimation requires recovering the information needed for regression from the contaminated observations. In finite-dimensional settings, this is often approached through deconvolution, whereas for Banach space-valued predictors the usual Fourier-based techniques may not be directly available. Our approach instead treats the measurement-error mechanism through an operator framework and focuses on constructing inverse weights that reproduce, exactly or approximately, the kernel weights that would have been available from the uncontaminated predictor. This formulation highlights the central roles of identifiability, stability, and regularization, particularly when exact inversion is unavailable or ill-posed. We discuss this inverse-problem perspective and the corresponding construction of inverse weights in Section \ref{sec:invWeight}.

Finally, we consider partial contamination models in Section~\ref{sec:partial}, where the measurement-error distribution is a mixture of a point mass at zero and a nondegenerate contamination distribution. Such models provide an intermediate setting between completely observed and fully contaminated Banach-valued predictors and arise naturally when complex functional objects are assembled from measurements of differing reliability.

In Section \ref{sec:disc}, we conclude the paper with an overall discussion of the proposed framework, its implications, and potential directions for future research. For completeness and to maintain the flow of the main presentation, proofs of all theoretical results established throughout the paper are deferred to the Appendix.

\section{Methodology}
\label{sec:method}
Let $(\mathcal{B},\|\cdot\|_{\mathcal{B}})$ be a real separable Banach space. Let $X$ be a $\mathcal{B}$-valued latent predictor and let $Y\in\mathbb{R}$ satisfy $\mathbb{E}|Y|^2<\infty$. In this article, we study the estimation of the nonparametric regression model $Y = m(X) + \varepsilon$, where $Y: \Omega \rightarrow \bbR$ is the response variable, $X: \Omega \rightarrow \sB$ is the predictor, and the error $\varepsilon: \Omega \rightarrow \bbR$ is an error term satisfying the condition $\E\{\varepsilon|X\} = 0$. Here $\Omega$ denotes the underlying sample space and $m: \sB \rightarrow \bbR$ is a fixed Borel version of the unknown regression function to be estimated. 
\par
In addition to the preceding setup, we assume that $X$ is not observed directly; instead, we observe an error-contaminated surrogate $W = X + U$ where $U: \Omega \rightarrow \sB$ is an unobserved, Banach space-valued random measurement error that is statistically independent of the pair $(X, Y)$. Therefore, the data consist of independent copies $\{(Y_i, W_i): i = 1, \cdots, n\}$ of $(Y, W)$. Throughout, $P_X$, $P_U$, and $P_W$ denote the distributions of $X$, $U$, and $W$, respectively. We assume that $P_U$ is known. Additional conditions will be introduced as needed throughout the paper for these distributions.
\par
For a bounded kernel $K:[0,\infty)\to[0,\infty)$ and bandwidth $h>0$, define $K_{h,x}(z) = K\left(\frac{\|z-x\|_{\mathcal{B}}}{h}\right)$. If $X_i$ were observed, the local-constant population target would be
\begin{equation}
m_h(x) = \frac{\mathbb{E}\{m(X)K_{h,x}(X)\}}{\mathbb{E} \{K_{h,x}(X)\}}.    
\end{equation}
Under local regularity of $m$, $m_h(x)$ approaches $m(x)$ as $h\downarrow0$.
Replacing $X$ by $W$ instead targets a local regression based on the contaminated predictor and does not, in general, recover $m(x)$.

For every bounded Borel function $L:\mathcal{B}\to\mathbb{R}$, define the measurement-error operator
\begin{equation}
    (\mathcal{A}_U L)(z) = \int_{\mathcal{B}}L(z+u)\,P_U(du).
\end{equation}
Independence of $U$ and $(X,Y)$ gives the central identities
\begin{equation}\label{eq:conditional-action}
\E\{L(W)\mid X\}=(\mathcal A_U L)(X),
\qquad
\E\{YL(W)\}
   =\E\{m(X)(\mathcal A_U L)(X)\}.
\end{equation}
For a given \(h>0\) and \(x\in\mathcal{B}\), suppose that there
exists a bounded Borel-measurable function
\(L_{h,x}:\mathcal{B}\to\mathbb{R}\) satisfying
\begin{equation}\label{eq:exact-inverse}
    (\mathcal{A}_U L_{h,x})(z)
    =
    K_{h,x}(z)
    \qquad \text{for }P_X\text{-almost every }z\in\mathcal{B}.
\end{equation}
In infinite-dimensional inverse problems, the operator equation \eqref{eq:exact-inverse} may be ill-posed: an exact solution need not exist, and even when a solution exists, it may fail to depend continuously on the target $K_{h,x}$; see \citet{CarrascoFlorensRenault2007}. We therefore use a regularized approximate inverse $L_{\lambda,h,x}$ and define its residual $R_{\lambda,h,x}$ by
\begin{equation}\label{eq:approximate-inverse}
(\mathcal{A}_U L_{\lambda,h,x})(z) = K_{h,x}(z)+R_{\lambda,h,x}(z).
\end{equation}

\noindent Write
\[\widehat{N}_{\lambda,h}(x) = \frac{1}{n}\sum_{i=1}^nY_iL_{\lambda,h,x}(W_i), \quad \widehat{D}_{\lambda,h}(x) = \frac{1}{n}\sum_{i=1}^nL_{\lambda,h,x}(W_i). \]
\noindent We may also write $\widehat D_n^L(x):=\widehat D_{\lambda_n,h_n}(x).$ The resulting estimator is 

\begin{equation}\label{eq:estimator}
\widehat m_{\lambda,h}(x)=
\begin{cases}
\widehat N_{\lambda,h}(x)/\widehat D_{\lambda,h}(x),&
\widehat D_{\lambda,h}(x)\ne0,\\[2mm]
0,&\widehat D_{\lambda,h}(x)=0.
\end{cases}
\end{equation}
For later use, define the following notations. 
\begin{notation}
\label{eq:residual}
For a fixed bandwidth $h>0$ and a point $x\in\mathcal{B}$, we introduce the following notation for the population-level denominator, numerator, and the corresponding smoothed regression function: $D_h(x)=\mathbb{E}\{K_{h,x}(X)\}, 
N_h(x)=\mathbb{E}\{m(X)K_{h,x}(X)\},
m_h(x)=\frac{N_h(x)}{D_h(x)}$
provided that $D_h(x)>0$. 
We also write $D^L_{\lambda,h}(x)=\mathbb{E}\{L_{\lambda,h,x}(W)\}$, $N^L_{\lambda,h}(x)=\mathbb{E}\{YL_{\lambda,h,x}(W)\}$ and $m^L_{\lambda,h}(x)=\frac{N^L_{\lambda,h}(x)}{D^L_{\lambda,h}(x)}$, provided $D^L_{\lambda,h}(x)\ne0$. 
\end{notation}
\begin{notation}
\label{eq:s-definitions}
    We also introduce notation for the population-level quantities associated with the residual kernels. For $\lambda>0$, $h>0$, and $x\in\mathcal{B}$, define
    \begin{align*}
r_{0,\lambda,h}(x)&=\mathbb{E}\{R_{\lambda,h,x}(X)\}, & r_{1,\lambda,h}(x)&=\mathbb{E}\{m(X)R_{\lambda,h,x}(X)\},\label{eq:r-definitions}\\
s^2_{0,\lambda,h}(x)&=\operatorname{Var}\{L_{\lambda,h,x}(W)\}, & s^2_{1,\lambda,h}(x)&=\operatorname{Var}\{YL_{\lambda,h,x}(W)\}.
\end{align*}
\end{notation}
\begin{notation}
    Fix $x\in\mathcal{B}$. Let $\{h_n\}_{n\geq1} \downarrow 0$ and $\{\lambda_n\}_{n\geq1}$ be positive sequences such that $h_n \downarrow 0$
    and $\lambda_n \downarrow 0$. Further define $L_n:=L_{\lambda_n,h_n,x}$ and $R_n:=R_{\lambda_n,h_n,x}$.
\end{notation}

\section{Results}
\label{sec:results}

In this section, we develop the theoretical foundation for the subsequent developments in Section \ref{sec:example}. We begin by introducing a first set of assumptions under which we establish the initial theoretical properties of the proposed framework. We discuss the mathematical and statistical roles of these assumptions, together with their practical interpretations and implications for the resulting theory. Following these results and the accompanying discussion, we introduce a second set of assumptions tailored to the additional theoretical questions considered later in this section. These two sets of conditions therefore serve distinct purposes and provide the basis for the corresponding theoretical and inferential developments that follow.

Assume the following.

\begin{enumerate}[label=(A\arabic*)]
\item\label{cond:K} The function $K:[0,\infty)\to[0,\infty)$ is bounded, Borel measurable, is not identically zero, and satisfies $\operatorname{supp}(K)\subseteq[0,1]$.

\item\label{cond:D} For all sufficiently large $n$, $D_{h_n}(x)>0$.

\item\label{cond:m} There exist $\delta>0$, $\alpha\in(0,1]$, and $L_m<\infty$ such that
\[
|m(z)-m(x)|\le L_m\|z-x\|_{\mathcal{B}}^\alpha
\]
whenever $\|z-x\|_{\mathcal{B}}\le\delta$. Moreover, for some $M_m<\infty$, $|m(X)|\le M_m$ almost surely.

\item\label{cond:residual} The approximate-inverse residual satisfies
\[
\frac{|r_{0,\lambda_n,h_n}(x)|+|r_{1,\lambda_n,h_n}(x)|}{D_{h_n}(x)}\longrightarrow0, \text{as } n \rightarrow \infty.
\]

\item\label{cond:L} For every $n$, $L_n(W)\in L^2$, $YL_n(W)\in L^2$, and
\[
\frac{s_{0,\lambda_n,h_n}(x)+s_{1,\lambda_n,h_n}(x)}{\sqrt{n}\,D_{h_n}(x)}\longrightarrow0 , \text{as } n \rightarrow \infty.
\]
\end{enumerate}

Assumptions \ref{cond:K}--\ref{cond:m} impose standard regularity requirements commonly used in classical nonparametric analysis and in the literature on nonparametric regression. 
Based on Assumption \ref{cond:K}, we impose the boundedness and compact support requirement on the kernel $K$. Restricting the support of $K$ to $[0,1]$ ensures that an observation receives nonzero weight only when $\|z-x\|_{\sB} \leq h$. This gives the estimator its local character: observations lying outside the $h$-neighborhood of $x\in \sB$ make no contribution to the corresponding kernel average. From a theoretical perspective, boundedness prevents any single kernel weight from becoming arbitrarily large and facilitates uniform control of the expectations, variances, and stochastic fluctuations of the kernel-weighted terms. Compact support further localizes these quantities to a shrinking neighborhood of $x$, thereby allowing their asymptotic behavior to be described in terms of the corresponding small-ball probabilities. In particular, the condition yields the elementary bound $0\leq K\left(\frac{\|z-x\|_{\sB}}{h}\right) \leq \lVert K\rVert_{\infty}\mathbf{1}\{\|z-x\|_{\sB}\leq h\}$, which is repeatedly useful for controlling kernel moments and deriving convergence rates. Kernels satisfying this or closely related conditions are routinely employed in nonparametric regression with functional predictors and in related functional-data problems; see, for example, \citet{ferratymasvieu2007, FerratyVanKeilegomVieu2012, geenens2011curse, chowdhury2019nonparametric, zhu2017kernel}. 

Due to Assumption \ref{cond:D}, $D_{h_n}(x)$ measures the amount of probability mass assigned by the distribution of $X$ to an $h_n$-neighborhood of $x$, weighted according to the kernel. Thus, the requirement $D_{h_n}(x)>0$ for all sufficiently large $n$ guarantees that, even as the bandwidth shrinks to zero, the target point $x$ remains locally accessible under the distribution of the functional covariate. This condition is necessary for the population-level smoothed regression function. Moreover, H\"older-type smoothness Assumption \ref{cond:m} are routinely imposed to control approximation errors in nonparametric estimation and functional-data problems. See \citet{lian2011convergence, chen2011single, delsol2011structural, chagny2016adaptive} for more details. 

Assumption \ref{cond:residual} controls the bias introduced by replacing the exact inverse by an approximate inverse. Indeed, $r_{0,\lambda_n,h_n}(x)$ and $r_{1,\lambda_n,h_n}(x)$ are the effects of the residual on the population denominator and numerator, respectively. Because the estimator is a ratio whose population denominator is of order $D_{h_n}(x)$, these residual terms must be negligible relative to $D_{h_n}(x)$. In particular, since $|m(X)|\le M_m$ almost surely, $|r_{0,\lambda_n,h_n}(x)| + |r_{1,\lambda_n,h_n}(x)| \le (1+M_m) \|R_{\lambda_n,h_n,x}\|_{L^2(P_X)}$. Consequently, a sufficient condition for Assumption \ref{cond:residual} is
\[
\frac{\|R_{\lambda_n,h_n,x}\|_{L^2(P_X)}}{D_{h_n}(x)} \longrightarrow0, \text{ as } n \rightarrow \infty.
\]
For an exact inverse, $R_{\lambda_n,h_n,x}=0$, so Assumption \ref{cond:residual} holds automatically.

Assumption \ref{cond:L} controls the stochastic fluctuations of the numerator and denominator. Their standard deviations are of orders $\frac{s_{1,\lambda_n,h_n}(x)}{\sqrt{n}}$ and $\frac{s_{0,\lambda_n,h_n}(x)}{\sqrt{n}}$, respectively. Dividing these fluctuations by the population denominator $D_{h_n}(x)$ gives the normalization in Assumption \ref{cond:L}. Thus, this assumption ensures both that the sampling error vanishes and that the denominator remains separated from zero with probability tending to one. If, in addition, $\mathbb{E}(Y^2\mid W)\le M_Y$ almost surely, then $s_{0,\lambda_n,h_n}(x) + s_{1,\lambda_n,h_n}(x) \le (1+\sqrt{M_Y}) \|L_{\lambda_n,h_n,x}\|_{L^2(P_W)}$.
Hence, a sufficient condition for Assumption \ref{cond:L} is
\[
\frac{\|L_{\lambda_n,h_n,x}\|_{L^2(P_W)}}{\sqrt{n}\,D_{h_n}(x)} \longrightarrow0 \text{ as } n \rightarrow \infty.
\]

Assumptions \ref{cond:residual} and \ref{cond:L} therefore express the usual regularization trade-off: improving the approximation to $K_{h_n,x}$ reduces the residual in Assumption \ref{cond:residual}, but may increase the norm and variability of the inverse weight, making Assumption \ref{cond:L} more difficult to satisfy.

The following result brings together the main components required to control the pointwise estimation error of the proposed estimator. First, it characterizes the discrepancy between the regularized population numerator and denominator and their uncontaminated kernel counterparts through the remainder terms $r_{0,\lambda_n,h_n}(x)$, $r_{1,\lambda_n,h_n}(x)$, respectively. It then controls the smoothing bias using the local H\"older regularity of the regression function. Finally, it establishes the asymptotic stability of both the empirical and regularized population denominators and incorporates the sampling-error bound derived earlier.

\begin{lemma}\label{lem:technical-details}
Under Assumptions \ref{cond:K}--\ref{cond:L}, the following statements hold.
\begin{enumerate}[label=(\roman*)]
\item For every $(\lambda,h)$ for which the expectations exist,
\begin{equation}\label{eq:population-identities}
D^L_{\lambda,h}(x)=D_h(x)+r_{0,\lambda,h}(x), \qquad N^L_{\lambda,h}(x)=N_h(x)+r_{1,\lambda,h}(x).
\end{equation}

\item For all sufficiently large $n$, $|m_{h_n}(x)-m(x)|\le L_mh_n^\alpha.$

\item
\begin{equation}\label{eq:den-relative}
\frac{\widehat{D}_{\lambda_n,h_n}(x)}{D_{h_n}(x)} \overset{\bbP}{\longrightarrow}1, \qquad \frac{D^L_{\lambda_n,h_n}(x)}{D_{h_n}(x)}\longrightarrow1.
\end{equation}
In particular, the following holds.
\begin{equation}\label{eq:den-positive}
\mathbb{P}\!\left\{\widehat{D}_{\lambda_n,h_n}(x) \ge \frac{1}{2}D_{h_n}(x)\right\}\longrightarrow1,
\end{equation}

\item The sampling error satisfies \begin{equation}\label{eq:sampling-bound}
\widehat{m}_{\lambda_n,h_n}(x)-m^L_{\lambda_n,h_n}(x)
=O_{\mathbb{P}}\!\left(
\frac{s_{0,\lambda_n,h_n}(x)+s_{1,\lambda_n,h_n}(x)}
{\sqrt{n}\,D_{h_n}(x)}\right).
\end{equation}
\end{enumerate}
\end{lemma}

The following theorem provides an explicit upper bound for the estimation error $\left|\widehat{m}_{\lambda_n,h_n}(x)-m(x)\right|$, thereby quantifying the accuracy of the proposed estimator at a fixed point $x$. In particular, the bound clarifies how the estimation error is governed by the smoothing parameter $h_n$, the regularization parameter $\lambda_n$, the local behavior of the regression function, and the stochastic variability arising from the observed data. Consequently, the result not only establishes pointwise consistency under suitable choices of $\lambda_n$ and $h_n$, but also provides theoretical guidance for balancing the different sources of error when selecting these tuning parameters.

\begin{theorem}\label{thm:main}
Under Assumptions \ref{cond:K}--\ref{cond:L},
\begin{align}
\left|\widehat{m}_{\lambda_n,h_n}(x)-m(x)\right|
&=O_{\mathbb{P}}\Bigg[h_n^\alpha
+\frac{|r_{0,\lambda_n,h_n}(x)|+|r_{1,\lambda_n,h_n}(x)|}
{D_{h_n}(x)}+\frac{s_{0,\lambda_n,h_n}(x)+s_{1,\lambda_n,h_n}(x)}
{\sqrt{n}\,D_{h_n}(x)}\Bigg].\label{eq:main-rate}
\end{align}
Consequently,
\begin{equation}\label{eq:consistency}
\widehat{m}_{\lambda_n,h_n}(x)\overset{\bbP}{\longrightarrow}m(x).
\end{equation}
Moreover,
\begin{equation*}
\mathbb{P}\!\left\{\widehat{D}_{\lambda_n,h_n}(x) \ge \frac{1}{2}D_{h_n}(x)\right\}\longrightarrow1,
\end{equation*}
so the estimator in \eqref{eq:estimator} is well defined with probability tending to one.
\end{theorem}

For all sufficiently large $n$, Lemma~\ref{lem:technical-details} gives $D^{L}_{\lambda_n, h_n}(x)>0$ and define
\begin{equation}\label{eq:sigma-definition}
\sigma_n^2(x) = \frac{\mathbb{E}\!\left[\{Y-m^L_{\lambda_n,h_n}(x)\}^2L_n^2(W)\right]}{D^L_{\lambda_n,h_n}(x)}.
\end{equation}
The same argument as in the proof of Lemma~\ref{lem:technical-details} gives the pointwise stochastic order of the estimator around its population-level regularized target
\begin{equation}\label{eq:sigma-form}
\widehat{m}_{\lambda_n,h_n}(x)-m^L_{\lambda_n,h_n}(x) = O_{\mathbb{P}}\!\left(\frac{\sigma_n(x)}{\sqrt{nD^L_{\lambda_n,h_n}(x)}}\right).
\end{equation}

Equivalently, the normalized sequence $\frac{\sqrt{nD^L_{\lambda_n,h_n}(x)}}
{\sigma_n(x)}
\left\{
\widehat m_{\lambda_n,h_n}(x)
-
m^L_{\lambda_n,h_n}(x)
\right\}$
is bounded in probability. Thus, the stochastic accuracy is governed
jointly by the weighted residual variability $\sigma_n^2(x)$ and the
normalizing quantity $D^L_{\lambda_n,h_n}(x)$. A larger value of
$\sigma_n^2(x)$ increases the sampling fluctuation, whereas a larger
positive value of $D^L_{\lambda_n,h_n}(x)$ improves the stochastic
precision through the factor $\sqrt{nD^L_{\lambda_n,h_n}(x)}$. The following lemma establishes the required variance-negligibility condition.

\begin{lemma}
\label{lemma:sigma}
    Under the Assumptions \ref{cond:K}--\ref{cond:L}, as $n \rightarrow \infty$, $\dfrac{\sigma_n^2(x)}
{\{nD^L_{\lambda_n,h_n}(x)\}}  \longrightarrow 0$.
\end{lemma}

The above lemma is sufficient for $\widehat m_{\lambda_n,h_n}(x)
-
m^L_{\lambda_n,h_n}(x)
\overset{\bbP}{\longrightarrow}0$. Now, this conclusion concerns convergence to the regularized population
target $m^L_{\lambda_n,h_n}(x)$, rather than directly to the true
regression function $m(x)$. Indeed, the total pointwise estimation
error admits the decomposition
\[
\widehat m_{\lambda_n,h_n}(x)-m(x)
=
\left\{
\widehat m_{\lambda_n,h_n}(x)
-
m^L_{\lambda_n,h_n}(x)
\right\}
+
\left\{
m^L_{\lambda_n,h_n}(x)-m(x)
\right\},
\]
where the first term is the stochastic error quantified by \eqref{eq:sigma-form} and the second is the regularization and smoothing bias. Therefore, pointwise consistency for $m(x)$ follows once the stochastic scale tends to zero, and the population target satisfies $m^L_{\lambda_n,h_n}(x)\to m(x)$. We establish the pointwise asymptotic normality of the proposed estimator in Theorem \ref{thm:pointwise-asymptotic-normality} in Section \ref{sec:cb-point}, which provides the theoretical foundation for constructing pointwise confidence intervals for $m(x)$ where $x\in \sB$.

We next strengthen the preceding pointwise convergence result by establishing a uniform error bound over a totally bounded subset $\mathcal X$ of $\mathcal B$. The principal additional difficulty is that both the numerator and denominator processes must now be controlled simultaneously over $x\in\mathcal X$. In particular, the population
denominator must remain uniformly nondegenerate, while the empirical fluctuations of the kernel-weighted numerator and denominator must be uniformly negligible relative to this normalization. The resulting bound separates the total uniform estimation error into a stochastic component and a population approximation component. For the uniform analysis, we first introduce the following notations. 
\begin{notation}
For an integrable function $g$ of $(Y,W)$, we write 
\[\mathbb P g:=\mathbb E\{g(Y,W)\} \quad\; \text{and}\quad \mathbb P_n g:=\frac{1}{n}\displaystyle\sum_{i=1}^n g(Y_i,W_i),\] 
and, analogously, write $N_n^L(x):=\E\{YL_{n,x}(W)\}$ for the population numerator and
\[
\widehat N_n^L(x):=\mathbb P_n\{YL_{n,x}(W)\}=\frac1n\sum_{i=1}^nY_iL_{n,x}(W_i)
\]
for its empirical counterpart.
\end{notation}

\noindent The uniform stochastic conclusion uses the following conditions:

\begin{enumerate}[label=\textup{(U\arabic*)},leftmargin=2.4em]

\item\label{unif:measurability}
The set $\mathcal X$ is nonempty and totally bounded under
$\|\cdot\|_{\mathcal B}$.  For every $n$, the functions
$\{L_{n,x}:x\in\mathcal X\}$ are deterministic Borel functions, and
$L_{n,x}(W)$ and $YL_{n,x}(W)$ are integrable for every
$x\in\mathcal X$.  All function classes below are pointwise measurable, or have pointwise-separable versions.  Assume also that
every displayed supremum involving the estimator is measurable.  These
conditions ensure that all probabilities and stochastic orders below
are ordinary probabilities rather than outer probabilities.

\item\label{unif:kernel-holder}
The kernel $K:[0,\infty)\to[0,\infty)$ is Borel measurable and bounded,
and
$\operatorname{supp}(K)\subseteq[0,1]$.  For all sufficiently large
$n$, $D_{h_n}(x)>0, \,\, \text{for every }x\in\mathcal X.$
Moreover, $\E|m(X)|<\infty$, so the numerators defining
$m_{h_n}(x)$ are finite.
There are constants $\delta>0$, $\alpha\in(0,1]$, and
$L_m<\infty$, independent of $x$, such that $|m(z)-m(x)|
\le L_m\|z-x\|_{\mathcal B}^{\alpha}$
whenever $x\in\mathcal X$ and
$\|z-x\|_{\mathcal B}\le\delta$.

\item\label{unif:population-denominator}
For all sufficiently large $n$, $D_{\lambda_n,h_n}^L(x)>0$ for every
$x\in\mathcal X$, and $d_n
:=
\inf_{x\in\mathcal X}D_{\lambda_n,h_n}^L(x)>0$.

\item\label{unif:bernstein}
For all sufficiently large $n$ and $j\in\{0,1\}$, define
\begin{align*}
f_{0,n,x}(y,w)&:=L_{n,x}(w),\\
f_{1,n,x}(y,w)&:=
\{y-m^L_{\lambda_n,h_n}(x)\}L_{n,x}(w),
\end{align*}
and
\[
Z_{j,n,x}
:=
f_{j,n,x}(Y,W)-\mathbb{P} f_{j,n,x}.
\]
There are deterministic sequences $v_{j,n}>0$ and $b_{j,n}>0$ such
that, for every integer $q\ge2$,
\[
\sup_{x\in\mathcal X}\E|Z_{j,n,x}|^q
\le
\frac{q!}{2}v_{j,n}^2b_{j,n}^{q-2},
\qquad j\in\{0,1\}.
\]

\item\label{unif:modulus}
For all sufficiently large $n$ and $j\in\{0,1\}$, there are
$\gamma_j\in(0,1]$, a nonnegative
integrable random variable $A_{j,n}=A_{j,n}(Y,W)$, and a deterministic
number $c_{j,n}>0$ such that, outside a single null set,
\[
|f_{j,n,x}(Y,W)-f_{j,n,z}(Y,W)|
\le
A_{j,n}(Y,W)\|x-z\|_{\mathcal B}^{\gamma_j}
\]
for every $x,z\in\mathcal X$, and $\mathbb{P} A_{j,n}\le c_{j,n}$.

\item\label{unif:tuning}
There is a sequence $\eta_n\downarrow0$ such that, with
\begin{align*}
M_n
&:=N(\eta_n,\mathcal X,\|\cdot\|_{\mathcal B}),\\
\ell_n
&:=\log(2M_n),\\
a_{jn}
&:=
v_{j,n}\sqrt{\frac{\ell_n}{n}}
+b_{j,n}\frac{\ell_n}{n}
+c_{j,n}\eta_n^{\gamma_j},
\qquad j\in\{0,1\},
\end{align*}
one has $\frac{a_{0n}}{d_n}\longrightarrow0$.
Here $N(\eta,\mathcal X,\|\cdot\|_{\mathcal B})$ denotes the smallest
cardinality of an $\eta$-net of $\mathcal X$.  It is finite because
$\mathcal X$ is totally bounded.
\end{enumerate}

The first part of Assumption \ref{unif:measurability} controls the geometric complexity of the index set $\mathcal X$. Total boundedness ensures that, for every $\delta>0$, $\mathcal X$ can be covered by finitely many balls of radius
$\delta$, or equivalently, $N(\delta,\mathcal X,\|\cdot\|_{\mathcal B})<\infty$, where $N(\delta,\mathcal X,\|\cdot\|_{\mathcal B})$ denotes the corresponding
covering number. This condition does not require $\mathcal X$ to be finite-dimensional, but ensures that it can be approximated by a finite collection of representative elements at any prescribed resolution, a property that is fundamental for uniform stochastic arguments. Related compactness and covering conditions are standard in uniform nonparametric estimation with functional or infinite-dimensional covariates; see, for example, \citet{ferraty2010rate}.  The integrability conditions on $L_{n,x}(W)$ and $YL_{n,x}(W)$ ensure that the weighted quantities defining the estimator and their population counterparts are well defined. This requirement is
particularly relevant when $L_{n,x}$ is obtained through an inverse construction, since inverse weights may otherwise become unstable or unbounded. Finally, the pointwise measurability and separability requirements
are standard empirical-process conditions; see, for example,
\citet{vanDerVaartWellner1996}. They ensure that suprema indexed by $\mathcal X$ are measurable random variables, allowing the convergence statements throughout the paper to be formulated in terms of ordinary probability rather than outer probability.

The Assumption \ref{unif:kernel-holder} imposes standard local regularity requirements for
kernel-based nonparametric regression. The boundedness and compact support of
the kernel ensure that estimation at a point $x$ is driven only by observations
lying within a local neighborhood of $x$, while preventing individual kernel
weights from becoming arbitrarily large. Conditions of this type are commonly
used in nonparametric regression with functional and infinite-dimensional
covariates; see, for example, \citet{ferratymasvieu2007}. The requirement
$D_{h_n}(x)>0$ guarantees that the predictor distribution places positive local
mass around each point at which estimation is performed, so that the
population-level kernel denominator is well defined. This condition is closely
related to the small-ball probability structure that plays a fundamental role
in nonparametric regression with infinite-dimensional covariates; see also
\citet{chowdhury2019nonparametric}. The integrability condition on $m(X)$ ensures
that the corresponding population numerators are finite. Finally, the local
H\"older condition controls the smoothness of the regression function uniformly
over $\mathcal X$. It formalizes the natural requirement that predictor values
that are close in the Banach-space norm should have similar conditional mean
responses and, in particular, provides the principal control on the local
smoothing bias of the kernel estimator.

The current Assumption \ref{unif:kernel-holder} is closely related to an earlier Assumption \ref{cond:m} imposed on the regression function $m$, but it is deliberately formulated in a weaker and more targeted form. In the earlier formulation, we assumed that there exist constants $\delta>0$, $\alpha\in(0,1]$, and $L_m<\infty$ such that $|m(z)-m(x)|
\leq
L_m\|z-x\|_{\mathcal B}^{\alpha}$, whenever $\|z-x\|_{\mathcal B}\leq\delta$, together with the global boundedness condition $|m(X)|\leq M_m$, almost surely for some finite constant $M_m$. The current formulation retains the same local H\"older-type smoothness requirement but imposes it only for $x\in\mathcal X$, namely, $|m(z)-m(x)|
\leq
L_m\|z-x\|_{\mathcal B}^{\alpha}$ where $x\in\mathcal X, \|z-x\|_{\mathcal B}\leq\delta$. This restriction is sufficient for the subsequent results because the estimation and uniform inference are carried out only over the target set $\mathcal X$. Hence, there is no need to impose the same local regularity condition around every point of the entire Banach space. In this sense, the current assumption is better aligned with the domain over which the theoretical results are established.

The Assumption \ref{unif:kernel-holder} also replaces the almost-sure boundedness of $m(X)$ by the weaker requirement $\mathbb E|m(X)|<\infty$. The relationship between the two assumptions is immediate: $|m(X)|\leq M_m$ almost surely implies $\mathbb E|m(X)|\leq M_m<\infty$, whereas the converse need not hold. Thus, almost-sure boundedness automatically implies the integrability condition, but integrability allows $m(X)$ to be unbounded as long as its first absolute moment is finite. Since the role of this condition in the present setting is primarily to guarantee that the population numerator defining $m_{h_n}(x)$ is finite, the weaker integrability assumption is sufficient and avoids imposing an unnecessarily restrictive global boundedness condition. Another important feature of Condition \ref{unif:kernel-holder} is that the same constants $\delta$, $L_m$, and $\alpha$ apply uniformly over $x\in\mathcal X$, yielding the uniform bias bound $|m(z)-m(x)|\le L_m h_n^\alpha$ whenever the kernel weight is nonzero. This uniformity is essential for the subsequent uniform convergence results.

Similar to Assumption \ref{cond:D}, Assumption \ref{unif:population-denominator} guarantees that the population denominator of the weighted estimator is well defined at every $x\in\mathcal X$, while $d_n=\inf_{x\in\mathcal X}D_{\lambda_n,h_n}^L(x)>0$ strengthens this requirement uniformly over the target set $\sX$. Thus, the effective local information available for estimation is prevented from degenerating at
any point of $\mathcal X$, which is essential for controlling the denominator in uniform convergence arguments.

The Assumption \ref{unif:modulus} imposes a uniform H\"older-type regularity requirement on the indexed random functions $f_{j,n,x}$, for $j\in\{0,1\}$. Specifically, it controls
their local oscillation in the index $x$ by a random envelope $A_{j,n}$ whose average magnitude is bounded by $c_{j,n}$. Conditions of this type are closely
related to the stochastic equicontinuity assumptions used in empirical-process theory and in the derivation of uniform convergence results for kernel-type estimators; see, for example, \citet{EinmahlMason2005, chernozhukov2014gaussian}. Together with the total boundedness of $\mathcal X$, this assumption allows the behavior of the empirical process over the entire index set to be controlled through a finite covering of $\mathcal X$ and bounds on the oscillations between nearby points. This condition should be distinguished from the H\"older condition imposed on the regression function $m$: the latter controls the deterministic smoothing bias, whereas the present condition controls the stochastic variation of the numerator- and denominator-type empirical processes and is therefore essential for the subsequent uniform convergence results.

The Assumption \ref{unif:tuning} balances the geometric complexity of the target set $\sX$, the
stochastic fluctuations of the empirical processes, and the stability of the
population denominator. The quantity
$M_n=N(\eta_n,\mathcal X,\|\cdot\|_{\mathcal B})$ measures the size of a finite
$\eta_n$-net of $\mathcal X$, while $\ell_n=\log(2M_n)$ represents the
corresponding complexity penalty. The first two terms in $a_{jn}$ control the
stochastic fluctuation over the finite net, whereas
$c_{j,n}\eta_n^{\gamma_j}$ controls the oscillation between nearby points
through the H\"older-type condition imposed above. Such covering-number and
entropy arguments are standard in uniform empirical-process and kernel
estimation theory; see, for example, \citet{EinmahlMason2005, van2011local}. Finally, the requirement
$a_{0n}/d_n\to0$ ensures that the uniform stochastic error of the denominator is
asymptotically negligible relative to its smallest population value. Thus,
this condition strengthens the earlier positivity requirement $d_n>0$ by
providing the rate separation needed for uniform stability of the ratio
estimator.

\section{Statistical inference and further methodological developments}
\label{sec:example}

This section illustrates how the proposed framework can be used to study several fundamental inferential problems and methodological extensions in nonparametric regression with Banach space-valued predictors with measurement error. 

\subsection{Pointwise inference}
\label{sec:cb-point}

Before stating the next theorem, we introduce some additional notation that will be used throughout the subsequent theoretical developments.

\begin{notation}
Fix $x\in\mathcal B$. Write $D_n^L=D_{\lambda_n,h_n}^L(x)=\E\{L_n(W)\}$, $m_n^L=m_{\lambda_n,h_n}^L(x)=\E\{YL_n(W)\}/D_n^L$, $\widehat D_n=\mathbb P_n L_n(W)$, and
\[
\widehat m_n=\frac{\mathbb P_n\{YL_n(W)\}}{\widehat D_n} \text{ whenever } \widehat D_n\ne0, \qquad \widehat m_n=0 \text{ otherwise}.
\]
Define the weighted residual $Q_n(Y,W)=\{Y-m_n^L\}L_n(W)$ and its variance $V_n=\E\{Q_n^2(Y,W)\}$; whenever $D_n^L>0$, set $\sigma_n^2=V_n/D_n^L$, so $V_n=D_n^L\sigma_n^2$. The quantities $L_n$, $D_n^L$, $m_n^L$, $\sigma_n$ depend on $x$, suppressed for simplicity.
\end{notation}

The following theorem establishes the pointwise asymptotic normality of the proposed estimator at a fixed $x\in \sX$, providing the basis for pointwise inference on $m(x)$.

\begin{theorem}[Pointwise asymptotic normality]
\label{thm:pointwise-asymptotic-normality}
Suppose that $\{(Y_i,W_i)\}_{i=1}^n$ are independent and identically distributed copies of $(Y,W)$ for each $n$. Fix nonrandom sequences $h_n,\lambda_n\to 0$, and let $L_n$ denote the corresponding inverse weight. Assume that, for all sufficiently large $n$, $D_n^L>0, 0<\sigma_n^2<\infty$. Suppose further that: 
\begin{enumerate}
    \item[(i)] $\widehat D_n$ is consistent for $D_n^L$: as $n\to\infty$,
\begin{equation*}
\frac{\operatorname{Var}\{L_n(W)\}}{n(D_n^L)^2}\longrightarrow 0;
\end{equation*}
\item[(ii)] for every $\varepsilon>0$, the Lindeberg condition holds:
\begin{equation}
\label{eq:lindeberg-condition}
\frac{
\E\!\left[
Q_n^2(Y,W) \mathbf 1\!\left\{ |Q_n(Y,W)| > \varepsilon\sqrt{nD_n^L}\,\sigma_n
\right\}
\right]
}{D_n^L\sigma_n^2}
\longrightarrow0.
\end{equation}
\end{enumerate}
Then
\begin{equation}
\label{eq:centered-clt}
\frac{\sqrt{nD_n^L}}{\sigma_n}
\{\widehat m_n-m_n^L\}
\overset{d}{\longrightarrow}N(0,1).
\end{equation}
If, in addition, 
\begin{equation}\label{eq:standardized-bias}
    \dfrac{\sqrt{nD_n^L}}{\sigma_n} \{m_n^L-m(x)\} \longrightarrow0,
\end{equation}
 then
\begin{equation}
\label{eq:true-target-clt}
\frac{\sqrt{nD_n^L}}{\sigma_n}
\{\widehat m_n-m(x)\}
\overset{d}{\longrightarrow}N(0,1).
\end{equation}
\end{theorem}

The pointwise asymptotic normality established in Theorem~\ref{thm:pointwise-asymptotic-normality}
provides a direct basis for statistical inference on the latent regression function at a fixed covariate value $x\in\mathcal B$. This is analogous in spirit to
pointwise inference in classical nonparametric regression and in nonparametric regression with functional covariates; see, for example,
\citet{ferratyvieu2006, ferratymasvieu2007}. The distinction in the present setting is that the predictor is allowed to take values in a general
Banach space and is not observed directly. Thus, the inferential target is the latent conditional mean $m(x)= \mathbb E\{Y\mid X=x\}$, rather than the regression function based on the contaminated predictor $W$.

For a fixed $x\in\mathcal B$, consider the pointwise hypothesis
\begin{equation}
\label{eq:pointwise-null}
    H_0:m(x)=m_0
    \qquad\text{versus}\qquad
    H_1:m(x)\neq m_0,
\end{equation}
where $m_0\in\mathbb R$ is a prespecified reference value. The formulation
\eqref{eq:pointwise-null} includes, as special cases, testing whether the
conditional mean vanishes at $x$, whether it equals a scientifically specified
reference level, or whether it equals a benchmark value obtained from an
external model or prior scientific knowledge. Theorem~\ref{thm:pointwise-asymptotic-normality} implies that, under the additional
negligibility condition
\[
    \frac{\sqrt{nD_n^L}}{\sigma_n}
    \{m_n^L-m(x)\}\longrightarrow 0,
\]
we have
\[
    \frac{\sqrt{nD_n^L}}{\sigma_n}
    \{\widehat m_n-m(x)\}
    \overset{d}{\longrightarrow}N(0,1).
\]
Consequently, under $H_0$,
\[
    \frac{\sqrt{nD_n^L}}{\sigma_n}
    \{\widehat m_n-m_0\}
    \overset{d}{\longrightarrow}N(0,1).
\]
To obtain a feasible test, it remains to replace the population normalizing
quantities by consistent estimators. In addition to the consistency of
$\widehat D_n$ established in Theorem~\ref{thm:pointwise-asymptotic-normality}, suppose
that an estimator $\widehat\sigma_n^2$ is available such that
\begin{equation}
\label{eq:variance-consistency}
    \frac{\widehat\sigma_n^2}{\sigma_n^2}
    \overset{\bbP}{\longrightarrow}1.
\end{equation}
As an immediate consequence of Theorem~\ref{thm:pointwise-asymptotic-normality} and Slutsky's theorem, we obtain the result below.

\begin{corollary}[Pointwise hypothesis testing]
\label{cor:pointwise-test}
Fix $x\in\mathcal B$ and $m_0\in\mathbb R$. Suppose that the conditions of
Theorem~\ref{thm:pointwise-asymptotic-normality} hold, including
\[
    \frac{\sqrt{nD_n^L}}{\sigma_n}
    \{m_n^L-m(x)\}\longrightarrow 0,
\]
and suppose further that \eqref{eq:variance-consistency} holds. Define $T_n(x;m_0)
    =
    \frac{\sqrt{n\widehat D_n}}
    {\widehat\sigma_n}
    \{\widehat m_n-m_0\}.$
Then, under the null hypothesis $H_0:m(x)=m_0$, $T_n(x;m_0)
    \overset{d}{\longrightarrow}N(0,1).$
Therefore, an asymptotic level-$\alpha$ test of
\eqref{eq:pointwise-null} rejects $H_0$ whenever $\left|T_n(x;m_0)\right|
    >
    z_{1-\alpha/2}$, 
where $z_q$ denotes the $q$th quantile of the standard normal distribution.
Equivalently, the test has asymptotic size $\alpha$, in the sense that $\bbP_{H_0}
    \left\{
        \left|T_n(x;m_0)\right|>z_{1-\alpha/2}
    \right\}
    \longrightarrow \alpha$.
\end{corollary}

The same result yields an asymptotic pointwise confidence interval for $m(x)$.
Specifically, a $100(1-\alpha)\%$ pointwise confidence interval is
\begin{equation}
\label{eq:pointwise-ci}
    \left[
    \widehat m_n
    -
    z_{1-\alpha/2}
    \frac{\widehat\sigma_n}{\sqrt{n\widehat D_n}},
    \;
    \widehat m_n
    +
    z_{1-\alpha/2}
    \frac{\widehat\sigma_n}{\sqrt{n\widehat D_n}}
    \right].
\end{equation}
Thus, the pointwise confidence interval and the test in
Corollary~\ref{cor:pointwise-test} are asymptotically dual, i.e., the null hypothesis
$H_0:m(x)=m_0$ is rejected at level $\alpha$ if and only if $m_0$ lies outside
the confidence interval in \eqref{eq:pointwise-ci}.

\subsubsection{Consistency under fixed alternatives}
The preceding test is also consistent against fixed alternatives under the
natural effective-rate condition. To see this, suppose that $m(x)=m_0+\delta$ for $\delta\neq 0.$ Then
\begin{align}
T_n(x;m_0)
&=
\frac{\sqrt{n\widehat D_n}}{\widehat\sigma_n}
\{\widehat m_n-m(x)\}
+
\frac{\sqrt{n\widehat D_n}}{\widehat\sigma_n}
\{m(x)-m_0\}
\nonumber\\
&=
\frac{\sqrt{n\widehat D_n}}{\widehat\sigma_n}
\{\widehat m_n-m(x)\}
+
\frac{\sqrt{n\widehat D_n}}{\widehat\sigma_n}\delta .
\label{eq:test-fixed-alt-decomp}
\end{align}
The first term in \eqref{eq:test-fixed-alt-decomp} is asymptotically standard
normal, whereas the magnitude of the second term diverges whenever
\begin{equation}
\label{eq:fixed-alt-rate}
    \frac{\sqrt{nD_n^L}}{\sigma_n}|\delta|
    \longrightarrow\infty.
\end{equation}
Consequently, $|T_n(x;m_0)|
    \overset{\bbP}{\longrightarrow}\infty,$
and hence $\bbP_{H_1}
    \left\{
        |T_n(x;m_0)|>z_{1-\alpha/2}
    \right\}
    \longrightarrow1.$
Condition \eqref{eq:fixed-alt-rate} highlights the effective pointwise
signal-to-noise scale of the problem. In particular, the ability to
distinguish $m(x)$ from $m_0$ is governed by the quantity $\frac{\sigma_n}{\sqrt{nD_n^L}}$, which incorporates both the local concentration of the latent covariate
through $D_n^L$ and the variability induced by the inverse weighting through
$\sigma_n$.

\subsubsection{Local alternatives and asymptotic power}
A more refined characterization of the sensitivity of the test is obtained by
considering alternatives approaching the null at the natural pointwise
estimation rate. For a fixed $c\in\mathbb R$, consider the sequence of local
alternatives
\begin{equation}
\label{eq:local-alt}
    H_{1,n}:
    m(x)
    =
    m_0+
    c\frac{\sigma_n}{\sqrt{nD_n^L}}.
\end{equation}
Under \eqref{eq:local-alt}, $\frac{\sqrt{nD_n^L}}{\sigma_n}
    \{m(x)-m_0\}
    =c.$
It follows from Theorem~\ref{thm:pointwise-asymptotic-normality},
the consistency of $\widehat D_n$ and $\widehat\sigma_n$, and Slutsky's theorem
that $ T_n(x;m_0)
    \overset{d}{\longrightarrow}N(c,1).$
Therefore, the limiting power of the two-sided level-$\alpha$ test is
\begin{align}
    \pi_\alpha(c)
    &=
    \Pr\left(
        |Z+c|>z_{1-\alpha/2}
    \right)
    \nonumber\\
    &=
    1-\Phi\left(z_{1-\alpha/2}-c\right)
    +
    \Phi\left(-z_{1-\alpha/2}-c\right),
\label{eq:local-power}
\end{align}
where $Z\sim N(0,1)$ and $\Phi$ denotes the standard normal distribution
function. Thus, the quantity $\frac{\sigma_n}{\sqrt{nD_n^L}}$ also determines the local separation rate at which alternatives can be
distinguished nontrivially from the null.

\subsubsection{Testing scientifically relevant deviations}
In some applications, testing exact equality to a reference value may be less
informative than determining whether the departure from that reference exceeds
a practically meaningful threshold. Let $\Delta>0$ denote such a threshold.
For example, a one-sided relevant-effect hypothesis can be formulated as
\begin{equation}
\label{eq:relevant-test}
    H_0:m(x)-m_0\leq\Delta
    \qquad\text{versus}\qquad
    H_1:m(x)-m_0>\Delta.
\end{equation}
The corresponding statistic is $ T_{n,\Delta}^{+}(x;m_0)
    =
    \frac{\sqrt{n\widehat D_n}}
    {\widehat\sigma_n}
    \{\widehat m_n-m_0-\Delta\}.$
At the least favorable boundary $m(x)=m_0+\Delta$, $T_{n,\Delta}^{+}(x;m_0)
\overset{d}{\longrightarrow}N(0,1)$, so that an asymptotic level-$\alpha$ test rejects
\eqref{eq:relevant-test} whenever $T_{n,\Delta}^{+}(x;m_0)>z_{1-\alpha}.$
Relevant hypotheses of this type have been advocated in functional-data
problems because they distinguish statistically detectable departures from
departures that exceed a prespecified practically relevant magnitude; see,
for example, \citet{dette2020functional}.

\begin{remark}[Interpretation in the Banach-space measurement-error setting]
\label{rem:pointwise-interpretation}
The preceding hypotheses concern the latent regression function $m(x)=\mathbb E\{Y\mid X=x\}$ evaluated at a fixed $x\in\mathcal B$. In particular, they do not concern the
conditional mean based on the contaminated predictor,
$\mathbb E\{Y\mid W=x\}$. This distinction is important because the inverse
weights used in $\widehat m_n$ are introduced precisely to recover information
about the regression relationship indexed by the unobserved $X$.

Moreover, no Hilbert-space representation, basis expansion, or finite-dimensional
projection of $x$ is required in the formulation of
\eqref{eq:pointwise-null}. The hypothesis itself is scalar, but the location at
which it is evaluated is an element of the potentially infinite-dimensional
Banach space $\mathcal B$. The geometry of $\mathcal B$ enters through the
localization mechanism and the corresponding quantity $D_n^L$, whereas the
inverse problem enters through the inverse weights and their contribution to
$\sigma_n$. In this sense, the test provides local inference directly on the
Banach-space regression surface.
\end{remark}

\begin{remark}[Relation to existing functional-regression tests]
\label{rem:testing-literature}
The pointwise problem considered here is distinct from the global structural
testing problems more commonly considered in the functional-regression
literature. For example, \citet{delsol2011structural} develop structural tests
for regression models with functional explanatory variables, whereas other
approaches test the global significance or no-effect of a functional covariate,
often through dimension reduction or projection arguments. The present
inference instead concerns the value of the latent conditional mean at a fixed
$x\in\mathcal B$ and follows directly from the pointwise asymptotic
normality of the inverse-weighted estimator. This local inference therefore
complements, rather than replaces, global goodness-of-fit or covariate-effect
tests. A simultaneous or global test over a set
$\mathcal X\subset\mathcal B$ would require stronger uniform approximations
than the pointwise limit theory used here.
\end{remark}

\subsection{Simultaneous inference}
\label{sec:cb-uniform}

The preceding results characterize the behavior of the proposed estimator at a fixed target point \(x\). We now strengthen this analysis by studying its behavior uniformly over a prescribed subset of the predictor space. Specifically, we derive a uniform bound for the smoothing and inverse-approximation bias and establish a uniform stochastic convergence rate for \(\widehat m(x)-m(x)\). These results quantify the maximal estimation error over the region of interest and provide the theoretical foundation for simultaneous inference, including uniform confidence bands and global hypothesis tests for the regression function.

\begin{theorem}[Uniform bias decomposition and uniform rate]
\label{thm:uniform-bias}
Suppose Conditions~\ref{unif:measurability}--\ref{unif:tuning} hold.
Define $\rho_n
:=
\sup_{x\in\mathcal X}
\left|
m^L_{\lambda_n,h_n}(x)-m_{h_n}(x)
\right|,$
and suppose $\rho_n<\infty$ for all sufficiently large $n$.  Then, for
every sufficiently large $n$ for which $h_n\le\delta$,
\begin{equation}\label{eq:uniform-population-bias}
\sup_{x\in\mathcal X}
\left|
m^L_{\lambda_n,h_n}(x)-m(x)
\right|
\le
\rho_n+L_mh_n^\alpha.
\end{equation}
Moreover,
\begin{equation}\label{eq:uniform-empirical-denominator-positive}
\mathbb{P}\left\{
\inf_{x\in\mathcal X}\widehat D_n^L(x)
\ge\frac{d_n}{2}
\right\}
\longrightarrow1,
\end{equation}
and
\begin{equation}\label{eq:uniform-total-rate}
\sup_{x\in\mathcal X}
\left|
\widehat m_{\lambda_n,h_n}(x)-m(x)
\right|
=
O_{\mathbb{P}}\!\left(
\rho_n+h_n^\alpha+\frac{a_{1n}}{d_n}
\right).
\end{equation}
In particular, if $\rho_n+h_n^\alpha+\frac{a_{1n}}{d_n}
\longrightarrow0,$ then $\widehat m_{\lambda_n,h_n}$ is uniformly consistent for $m$ on
$\mathcal X$.
\end{theorem}

\noindent The term $\rho_n$ can be expressed through the residuals as follows:

\begin{equation}
m^L_{\lambda_n,h_n}(x) - m_{h_n}(x)
= \frac{r_{1,\lambda_n,h_n}(x) - m_{h_n}(x)\, r_{0,\lambda_n,h_n}(x)}{D^L_{\lambda_n,h_n}(x)}.
\label{eq:uniform-exact-inverse-bias}
\end{equation}
Under Condition~\ref{unif:kernel-holder}, $\sup_{x \in \mathcal{X}}|m_{h_n}(x)|$ is bounded. Hence there exists a constant $C< \infty$ such that 
\[\rho_n \le \frac{\sup_{x \in \mathcal{X}}|r_{1,\lambda_n,h_n}(x)|+C\, \sup_{x \in \mathcal{X}}|r_{0,\lambda_n,h_n}(x)|}{d_n},\]
which is a condition that can be verified.
In addition to Assumptions~\ref{unif:measurability}--\ref{unif:tuning}, assume the following conditions. These conditions are high-level and do not follow from Conditions~\ref{unif:measurability}--\ref{unif:tuning}. We first introduce some notations. 
\begin{notation}
Restoring the dependence on $x$ in the pointwise notation, write
\[
m_n^L(x):=m^L_{\lambda_n,h_n}(x),\qquad
\widehat m_n(x):=\widehat m_{\lambda_n,h_n}(x).
\]
Let $f_{n,x}:=f_{1,n,x}$, with $f_{1,n,x}$ as defined in
Condition~\ref{unif:bernstein}. Thus
$f_{n,x}(y,w)=\{y-m_n^L(x)\}L_{n,x}(w)$ and
$\mathbb P f_{n,x}=0$. Define
\[
s_n^2(x):=\mathbb P f_{n,x}^2,\qquad
s_n(x):=\sqrt{s_n^2(x)},\qquad
\psi_{n,x}:=\frac{f_{n,x}}{s_n(x)}
\quad\text{when }s_n(x)>0.
\]
In the earlier pointwise notation,
$s_n^2(x)=D_n^L(x)\sigma_n^2(x)$.

For the given inverse weights, use
\[
\widehat s_n^2(x):=\frac1n\sum_{i=1}^n
\{Y_i-\widehat m_n(x)\}^2L_{n,x}^2(W_i),\qquad
\widehat s_n(x):=\sqrt{\widehat s_n^2(x)}.
\]
For a square-integrable function $g$, let
$\mathbb G_n g:=\sqrt n(\mathbb P_n g-\mathbb P g)$.
Write $\ell^\infty(\mathcal X)$ for the bounded real-valued
functions on $\mathcal X$, equipped with
$\|a\|_{\mathcal X}:=\sup_{x\in\mathcal X}|a(x)|$.
In particular,
$\|\mathbb G_n\psi_{n,\cdot}\|_{\mathcal X}
=\sup_{x\in\mathcal X}|\mathbb G_n\psi_{n,x}|$.
\end{notation}

We are now ready to introduce the other conditions: 
\begin{enumerate}[label=\textup{(U\arabic*)},start=7,leftmargin=2.4em]

\item\label{unif:gaussian-coupling}
For all sufficiently large $n$, $\inf_{x\in\mathcal X}s_n(x)>0.$ Moreover, there exists a tight Gaussian random variable $\mathbb{Z}_n^{G} \in \ell^{\infty}(\mathcal{X})$ with the following covariance function:
\begin{equation*}\label{eq:uniform-Gaussian-coupling}
\operatorname{Cov}
\{\mathbb Z_n^G(x),\mathbb Z_n^G(z)\}
=
\mathbb P\{\psi_{n,x}(Y,W)\psi_{n,z}(Y,W)\},
\qquad x,z\in\mathcal X.
\end{equation*}

\noindent Define $\mathfrak m_n:=1+\mathbb E\left\|\mathbb Z_n^G\right\|_{\mathcal X}.$ We assume that 
\begin{equation}
\left|
\left\|
\mathbb G_n\psi_{n,\cdot}
\right\|_{\mathcal X}
-
\left\|
\mathbb Z_n^G
\right\|_{\mathcal X}
\right|
=
o_{\mathbb P}(\mathfrak m_n^{-1}).
\end{equation}

\item\label{unif:undersmoothing}
The regularization and smoothing bias is negligible:
\begin{equation}
\mathfrak m_n
\sup_{x\in\mathcal X}
\frac{\sqrt n\,D_n^L(x)}{s_n(x)}
\left|
m_n^L(x)-m(x)
\right|
\longrightarrow0.
\label{eq:uniform-undersmoothing}
\end{equation}
\item\label{unif:big} The following two approximation conditions hold: 
\begin{enumerate}[label=\textup{(\alph*)},leftmargin=2.0em]
    \item The estimated score variance is uniformly consistent:
\begin{equation}
\mathfrak m_n^2 \sup_{x\in\mathcal X} \left| \frac{\widehat s_n^2(x)}{s_n^2(x)}-1\right| \overset{\mathbb P}{\longrightarrow}0.
\label{eq:uniform-score-variance-consistency}
\end{equation}

\item Let $\xi_1,\ldots,\xi_n$ be independent standard normal random
variables, independent of the data, and define
\begin{equation}
\widehat{\mathbb Z}_n^\xi(x)
:=
\frac{1}{\widehat s_n(x)\sqrt n}
\sum_{i=1}^n
\xi_i
\{Y_i-\widehat m_n(x)\}
L_{n,x}(W_i).
\label{eq:uniform-multiplier-process}
\end{equation}
Writing
\[
\mathcal D_n
:=
\sigma\{(Y_i,W_i):1\le i\le n\},
\]
the conditional distribution of the multiplier supremum satisfies
\begin{equation}
\sup_{t\in\mathbb R} \left| \mathbb P_\xi\left\{\left\| \widehat{\mathbb Z}_n^\xi \right\|_{\mathcal X}
\le t \,\middle|\, \mathcal D_n \right\} - \mathbb P \left\{ \left\|\mathbb Z_n^G \right\|_{\mathcal X}
\le t \right\} \right| \overset{\mathbb P}{\longrightarrow}0.
\label{eq:uniform-conditional-multiplier-approximation}
\end{equation}

\end{enumerate}
\end{enumerate}

Condition \ref{unif:gaussian-coupling} provides a uniform Gaussian approximation to the centered and standardized score process. The lower bound on \(s_n(x)\) prevents local degeneracy of the variance, while the coupling condition ensures that the supremum of the empirical process can be approximated by the supremum of a tight Gaussian process with the same covariance structure.  
Condition \ref{unif:undersmoothing} is a uniform undersmoothing condition. It requires the combined smoothing and regularization bias to be negligible relative to the stochastic scale of the studentized process, uniformly over \(x\in\mathcal X\). Such bias-negligibility conditions are standard in the construction of asymptotically valid uniform confidence bands, where the deterministic approximation error must vanish faster than the fluctuations governing the supremum statistic. Condition \ref{unif:big}(a) guarantees that studentization based on the estimated score variance is asymptotically innocuous, while Condition \ref{unif:big}(b) establishes the validity of the Gaussian multiplier bootstrap used to estimate critical values for the supremum distribution. Conditions of this general type are standard in modern uniform inference based on Gaussian approximation and multiplier bootstrap methods; see \citet{chernozhukov2014gaussian, CHERNOZHUKOV20163632}
Closely related ideas have been used specifically for uniform confidence bands in nonparametric errors-in-variables regression; see \citet{Kato2019uniform}.

\begin{remark}
By Theorem~\ref{thm:uniform-bias},
$\sup_{x\in\mathcal X} |m_n^L(x)-m(x)| \le \rho_n+L_mh_n^\alpha.$
Consequently, a sufficient condition for
\eqref{eq:uniform-undersmoothing} is
\begin{equation*}
\mathfrak m_n\sqrt n \left\{ \sup_{x\in\mathcal X} \frac{D_n^L(x)}{s_n(x)} \right\} \left(\rho_n+L_mh_n^\alpha
\right)\longrightarrow0.
\label{eq:primitive-uniform-undersmoothing}
\end{equation*}
This condition requires the combined regularization and smoothing bias to be asymptotically negligible relative to the simultaneous stochastic width of the estimator.
\end{remark}
\begin{remark}
    Condition \ref{unif:gaussian-coupling} is a high-level Gaussian-coupling condition for the standardized score class $\{\psi_{n,x}: x \in \mathcal{X}\}$. Condition \ref{unif:big}(b) is the high-level analogue of the conditional multiplier-bootstrap approximation proved in \cite[Theorem~3.2]{Kato2019uniform}. Their result concerns the deconvolution-kernel score process indexed by a subset of $\mathbb{R}$. In our setting, the score class is generated by the regularized inverse weights:
    \[\Big\{(Y,W) \mapsto \frac{\{Y-m_n^{L}(x)\}L_{n,x}(W)}{s_n(x)}: x \in \mathcal{X}\Big\},\] where $\mathcal{X}$ is a subset of a Banach space. Therefore, Theorem~3.2 of \cite{Kato2019uniform} cannot be applied directly without verifying the required Gaussian and multiplier approximations for this new score class. 
\end{remark}
Note that for a fixed $\tau\in(0,1)$, define the conditional multiplier critical
value
\begin{equation*}
\widehat c_n(1-\tau)
:=\inf\left\{t\in\mathbb R:\mathbb P_\xi \left(\left\|\widehat{\mathbb Z}_n^\xi\right\|_{\mathcal X}
\le t \,\middle|\, \mathcal D_n\right)\ge1-\tau\right\}.
\label{eq:uniform-bootstrap-critical-value}
\end{equation*}
The feasible simultaneous confidence band is
\begin{equation}
\widehat{\mathcal C}_{n,1-\tau}(x)
:=
\left[
\widehat m_n(x)
\pm
\frac{\widehat s_n(x)}
{\sqrt n\,\widehat D_n^L(x)}
\widehat c_n(1-\tau)
\right],
\qquad x\in\mathcal X.
\label{eq:feasible-uniform-confidence-band}
\end{equation}
\begin{theorem}
\label{thm:feasible-uniform-confidence-band}
Suppose Conditions \ref{unif:measurability}--\ref{unif:big} hold. Then
\begin{equation}
\mathbb P
\left\{
m(x)\in
\widehat{\mathcal C}_{n,1-\tau}(x)
\text{ for every }x\in\mathcal X
\right\}
=
1-\tau+o(1).
\label{eq:feasible-uniform-coverage}
\end{equation}
\end{theorem}

\subsubsection{Global hypothesis testing}

The simultaneous confidence band developed above can be used directly to construct global tests for the regression function over $\mathcal{X}$. Let $m_0:\mathcal{X}\to\mathbb{R}$ be a prespecified reference function. Consider the hypotheses

\begin{equation}
    H_0:m(x)=m_0(x)\quad \text{for every }x\in\mathcal{X},
\end{equation}
against
\begin{equation}
    H_1:m(x)\neq m_0(x)\quad \text{for at least one }x\in\mathcal{X}.
\end{equation}
A natural test statistic is the supremum of the studentized discrepancy between the estimator and the null regression function, $ T_n
=
\sup_{x\in\mathcal{X}}
\left|
\frac{
\sqrt{n}\,\widehat D_n^L(x)
}{
\widehat s_n(x)
}
\left\{
\widehat m_n(x)-m_0(x)
\right\}
\right|$.
Let $\widehat c_n(1-\tau)$ denote the conditional multiplier critical value defined above. We reject $H_0$ at asymptotic level $\tau$ whenever $T_n>\widehat c_n(1-\tau).$
Equivalently, by the duality between simultaneous confidence bands and global hypothesis tests, $H_0$ is rejected whenever the null function $m_0(\cdot)$ is not entirely contained in the simultaneous confidence band. 
Under $H_0$ and Conditions \ref{unif:measurability}--\ref{unif:big}, Theorem~\ref{thm:feasible-uniform-confidence-band} implies
\[
\mathbb P_{H_0}
\left\{
T_n>\widehat c_n(1-\tau)
\right\}
=
\tau+o(1).
\]
Hence, the proposed test has asymptotic size $\tau$.

This testing framework permits simultaneous assessment of the regression function over the entire predictor region $\mathcal{X}$ rather than at a single fixed predictor value. For example, taking $m_0(x)\equiv c$ yields a global test of whether the regression function is equal to a constant level $c$ throughout $\mathcal{X}$. More generally, $m_0(\cdot)$ may represent a scientifically specified reference curve or a regression function implied by a parametric model. Rejection indicates that the true regression function departs from the null specification at least somewhere over $\mathcal{X}$.

\subsection{Construction of approximate inverse weights}
\label{sec:invWeight}

An important feature of nonparametric regression with measurement error is its intrinsic connection with statistical inverse problems. When a predictor is observed with error, the data contain only a distorted version of the latent predictor of scientific interest, so the regression procedure must recover relevant information about the unobserved predictor from its contaminated measurements. This is fundamentally different from ordinary nonparametric regression, where the predictor entering the regression function is assumed to be observed directly. In finite-dimensional settings, this recovery problem is commonly formulated through deconvolution: the measurement-error mechanism smooths or obscures information about the latent predictor, and estimation requires inverting, either exactly or approximately, this distortion \citep{fan1991optimal,fan1993nonparametric}. Such inverse problems are typically ill-posed, in the sense that small perturbations in the observed distribution may lead to substantial changes in the recovered quantity, making identifiability, stability, and regularization central statistical considerations \citep{cavalier2008nonparametric}. The problem is particularly challenging for Banach space-valued predictors. Although probability measures on separable Banach space can be represented by characteristic functionals on the dual, however, there is no canonical locally finite, translation-invariant analogue of Lebesgue measure in infinite-dimensional separable Banach spaces. So the Fourier-transform-based inversion techniques that underlie classical Euclidean deconvolution are generally not directly available in this setting \citep{kuo1975, vakhania1987probability,DelaigleMeister2021}.

Our approach therefore represents the measurement-error mechanism more generally through an operator that maps functions evaluated at the contaminated predictor to corresponding functions of the latent predictor. Rather than attempting to reconstruct each unobserved predictor $X$ itself, we seek only to recover the weighting mechanism required for nonparametric regression. Specifically, for the kernel weight that would be used if $X$ were observed without error, we search for a function of the contaminated observation $W$ whose conditional action under the measurement-error mechanism reproduces that kernel weight. This leads naturally to the construction of an inverse weight and transforms the estimation problem into an operator inversion problem. Within this formulation, the existence of such an inverse weight is closely related to identification, whereas the stability of the inversion governs the feasibility and statistical accuracy of the resulting estimator. When exact inversion does not exist or is unstable, approximate inversion and regularization provide natural alternatives. This operator perspective therefore establishes a direct connection between nonparametric regression with measurement error and the broader theory of statistical inverse problems, while remaining applicable to general Banach space-valued random elements for which conventional Euclidean deconvolution methods are not readily available.

The construction below is formulated at the population level. Although the measurement-error distribution \(P_U\) is assumed to be known, the Hilbert-space formulation of the regularized inverse involves \(P_X\) through the space \(L^2(P_X)\) and the adjoint operator \(A^*\). Thus, the resulting Tikhonov weight should be interpreted as an oracle construction rather than as a directly computable data-based weight. Its role here is to establish the existence of stable approximate inverse weights and to provide sufficient conditions under which the abstract assumptions used in the preceding consistency result are nonempty. A feasible implementation would additionally require estimation or regularization of the latent distribution \(P_X\), which we do not pursue here.

Before stating the formal proposition, we introduce some notation that will be used throughout the subsequent analysis.

\begin{notation}
For a linear operator
$
A:L^2(P_W)\to L^2(P_X)
$, the range of \(A\) is defined as
$
\operatorname{Ran}(A)
=
\{AL:L\in L^2(P_W)\}
\subseteq L^2(P_X).
$ That is, \(\operatorname{Ran}(A)\) is the collection of all elements in \(L^2(P_X)\) that can be obtained by applying \(A\) to some \(L\in L^2(P_W)\).
The closure of the range is defined by
$$
\overline{\operatorname{Ran}(A)}
=
\left\{
K\in L^2(P_X):
\text{ there exists a sequence }
\{L_m\}_{m\ge1}\subset L^2(P_W)
\text{ such that }
\|AL_m-K\|_{L^2(P_X)}\to0
\right\}
$$
Therefore,
$
\overline{\operatorname{Ran}(A)}=L^2(P_X)
$ means that every \(K\in L^2(P_X)\)  can be approximated arbitrarily well, in the \(L^2(P_X)\) norm, by elements of the form \(AL\). Equivalently, for every \(K\in L^2(P_X)\), there exists a sequence \(\{L_m\}\subset L^2(P_W)\) such that
$
AL_m\to K$ in 
$L^2(P_X).
$
\end{notation}

Consider the bounded linear operator $
A:L^2(P_W)\longrightarrow L^2(P_X)$ defined by $
(AL)(X)=\mathbb{E}\{L(W)\mid X\}.
$ Since \(W=X+U\) and \(U\) is independent of \(X\),
$
(AL)(X)
=
\int_{\mathcal B}L(X+u)\,P_U(du)
=
(\mathcal A_U L)(X)$ almost surely. Thus, \(A\) provides the Hilbert-space representation of the measurement-error operator introduced above.

The following result shows that, under a nonvanishing characteristic-functional condition on the measurement error, the range of \(A\) is dense in \(L^2(P_X)\).

\begin{proposition}\label{prop:dense-range}
Let \(i=\sqrt{-1}\). Suppose that \(\mathcal B\) is a separable Banach space and that the characteristic functional of \(U\),
$
\phi_U(\ell)
=
\mathbb{E}\!\left[\exp\{i\ell(U)\}\right]$, for
$\ell\in\mathcal B^*$, satisfies
$ \phi_U(\ell)\neq0$
for every $\ell\in\mathcal B^*$. Then $
\overline{\operatorname{Ran}(A)}
=
L^2(P_X).
$
\end{proposition}

The nonvanishing characteristic-functional condition is satisfied by an important class of measurement-error distributions.
\begin{example}[Gaussian measurement error]
Suppose that \(U\) is a Gaussian random element of \(\mathcal B\) with mean \(\mu_U\). Then, for every \(\ell\in\mathcal B^*\),
$
\phi_U(\ell)
=
\exp\left\{
i\ell(\mu_U)
-
\frac{1}{2}\operatorname{Var}\{\ell(U)\}
\right\}.
$ Since the exponential function is never zero, $
\phi_U(\ell)\neq0$, $\ell\in\mathcal B^*,
$ and hence Proposition~\ref{prop:dense-range} applies.
\end{example}

For the general operator $A$, Proposition~\ref{prop:dense-range} implies that, for every \(K\in L^2(P_X)\) and every \(\varepsilon>0\), there exists an \(L\in L^2(P_W)\) such that $
\|AL-K\|_{L^2(P_X)}<\varepsilon.
$ Thus, every element of \(L^2(P_X)\), and in particular the target kernel weight \(K_{h,x}\), can be approximated arbitrarily well by elements in the range of \(A\), even when the exact inverse equation $AL=K_{h,x}$ does not admit a solution. Dense-range and related completeness conditions arise naturally in the identification of ill-posed inverse problems; see, for example, \cite{NeweyPowell2003,CarrascoFlorensRenault2007}. However, the density of the range alone provides no control over the magnitude of the approximating function \(L\). In an ill-posed problem, increasingly accurate approximation of \(K_{h,x}\) may require inverse weights with arbitrarily large \(L^2(P_W)\)-norm, leading to instability.

A standard approach for controlling this instability is Tikhonov regularization \cite{CarrascoFlorensRenault2007,DarollesFanFlorensRenault2011}. For an arbitrary target \(K\in L^2(P_X)\) and a regularization parameter \(\lambda>0\), consider the penalized inverse problem

$$
\min_{L\in L^2(P_W)}
\left\{
\|AL-K\|_{L^2(P_X)}^2
+
\lambda\|L\|_{L^2(P_W)}^2
\right\}.
$$

The first term measures the discrepancy between the action of the inverse weight and the desired target, whereas the second penalizes inverse weights with a large \(L^2(P_W)\)-norm. The parameter \(\lambda\) therefore controls the trade-off between approximation accuracy and stability.

Since \(A\) is bounded, the above criterion admits a unique minimizer given by
$$
L_\lambda
=
(A^*A+\lambda I)^{-1}A^*K
=
A^*(AA^*+\lambda I)^{-1}K, 
$$

where \(A^*:L^2(P_X)\to L^2(P_W)\) denotes the Hilbert-space adjoint of \(A\). Applying this construction to the target kernel weight \(K_{h,x}\), define
$L_{\lambda,h,x}
=
A^*(AA^*+\lambda I)^{-1}K_{h,x}$. The corresponding approximation residual is
$
R_{\lambda,h,x}
=
AL_{\lambda,h,x}-K_{h,x}.
$ Using the identity
$
AA^*(AA^*+\lambda I)^{-1}
=
I-\lambda(AA^*+\lambda I)^{-1},
$, 
we obtain
$
R_{\lambda,h,x}
=
-\lambda(AA^*+\lambda I)^{-1}K_{h,x}.
$

Moreover, by the spectral calculus for the nonnegative self-adjoint operator \(AA^*\),

$$
\|L_{\lambda,h,x}\|_{L^2(P_W)}
\le
\frac{\|K_{h,x}\|_{L^2(P_X)}}{2\sqrt{\lambda}}.
$$

For every fixed \(h>0\), Proposition~\ref{prop:dense-range} implies

$$
\|R_{\lambda,h,x}\|_{L^2(P_X)}
\longrightarrow 0
\qquad\text{as }\lambda\downarrow0.
$$

Hence, Tikhonov regularization provides a sequence of stable approximate inverse weights whose action approaches the target kernel weight as the regularization is relaxed. This conclusion is qualitative, however. In Theorem~\ref{thm:main}, the bandwidth \(h=h_n\) varies with the sample size, so the target \(K_{h_n,x}\) also changes with \(n\). Assumption \ref{cond:residual} requires the corresponding approximation residual to be negligible relative to the local normalization \(D_{h_n}(x)\). A quantitative bound on this residual is therefore needed.

To obtain such a bound, we impose a standard source condition. Suppose that, for some \(\nu\in(0,1]\), $K_{h,x}
=
(AA^*)^\nu g_{h,x}$ for some \(g_{h,x}\in L^2(P_X)\). This condition describes the regularity of the target \(K_{h,x}\) relative to the smoothing properties of the operator \(A\). Source conditions of this form are standard in the theory of regularized inverse problems and provide explicit rates for the approximation error; see \citet[Section 3]{CarrascoFlorensRenault2007},\citet{DarollesFanFlorensRenault2011}.

To make this condition more transparent, suppose for illustration that \(AA^*\) admits an orthonormal eigensystem \(\{(\mu_j,e_j)\}_{j\ge1}\), with
$
AA^*e_j=\mu_j e_j$ for
$\mu_j>0,
$
and write
$
K_{h,x}
=
\sum_{j\ge1}k_{j,h,x}e_j.
$
Then
$
K_{h,x}=(AA^*)^\nu g_{h,x}
$
holds for some \(g_{h,x}\in L^2(P_X)\) if
$
\sum_{j\ge1}
\frac{k_{j,h,x}^2}{\mu_j^{2\nu}}
<\infty.
$

In that case,
$
g_{h,x}
=
\sum_{j\ge1}
\frac{k_{j,h,x}}{\mu_j^\nu}e_j.
$
Thus, the source condition requires the target kernel weight to have sufficiently small components in directions corresponding to very small eigenvalues of \(AA^*\), which are precisely the directions along which inversion is most unstable.

Under this source condition,
$
R_{\lambda,h,x}
=
-\lambda(AA^*+\lambda I)^{-1}(AA^*)^\nu g_{h,x},
$ and spectral calculus yields
$$
\|R_{\lambda,h,x}\|_{L^2(P_X)}
\le
C_\nu\lambda^\nu
\|g_{h,x}\|_{L^2(P_X)},
$$

where

$$
C_\nu
=
\sup_{t\ge0}\frac{t^\nu}{1+t}
<\infty.
$$

This quantitative control of the regularization bias allows the general consistency result to be applied to the Tikhonov inverse weights.

\begin{corollary}\label{cor:tikhonov}
Suppose Assumptions \ref{cond:K}--\ref{cond:m} hold. Assume that, for some \(\nu\in(0,1]\) and all sufficiently large \(n\), there exists \(g_{h_n,x}\in L^2(P_X)\) such that
$
K_{h_n,x}
=
(AA^*)^\nu g_{h_n,x}.
$ Suppose also that, for some \(M_Y<\infty\),
$
\mathbb{E}(Y^2\mid W)\le M_Y
$ almost surely.
Let \(h_n\downarrow0\) and \(\lambda_n\downarrow0\) satisfy $\frac{
\lambda_n^\nu
\|g_{h_n,x}\|_{L^2(P_X)}
}{
D_{h_n}(x)
}
\longrightarrow0$
and $n\lambda_nD_{h_n}(x)
\longrightarrow\infty$. Define $L_{\lambda_n,h_n,x}
=
A^*(AA^*+\lambda_n I)^{-1}K_{h_n,x}.$
Then Assumptions \ref{cond:residual} and \ref{cond:L} hold. Consequently,
$$
\left|
\widehat m_{\lambda_n,h_n}(x)-m(x)
\right|
=
O_{\mathbb P}
\left[
h_n^\alpha
+
\frac{
\lambda_n^\nu
\|g_{h_n,x}\|_{L^2(P_X)}
}{
D_{h_n}(x)
}
+
\frac{1}{
\sqrt{n\lambda_nD_{h_n}(x)}
}
\right].
$$
In particular,
$$
\widehat m_{\lambda_n,h_n}(x)
\overset{\mathbb P}{\longrightarrow}m(x).
$$

\end{corollary}

The three terms in the convergence rate have distinct interpretations. The term \(h_n^\alpha\) is the usual local smoothing bias. The term $\frac{
\lambda_n^\nu
\|g_{h_n,x}\|_{L^2(P_X)}
}{
D_{h_n}(x)
}$ represents the approximation error induced by regularized inversion, whereas $\frac{1}{\sqrt{n\lambda_nD_{h_n}(x)}}$ captures the stochastic cost of using increasingly unstable inverse weights. The regularization parameter \(\lambda_n\) therefore controls the characteristic bias--stability trade-off of the inverse problem: decreasing \(\lambda_n\) improves the approximation of the uncontaminated kernel weight but increases the magnitude, and hence the sampling variability, of the corresponding inverse weight.

\begin{remark}[Computational caveat and a possible empirical construction]
The regularized inverse construction above is formulated in terms of the latent predictor \(X\), which is not directly observed. Although the measurement-error distribution \(P_U\) is assumed to be known, the observed data consist only of the contaminated measurements \(W=X+U\). Consequently, the distribution \(P_X\), the \(L^2(P_X)\) geometry, and the adjoint operator \(A^*\) appearing in $L_{\lambda,h,x}
=
A^*(AA^*+\lambda I)^{-1}K_{h,x}$ are not directly available from the observed sample. It is useful to distinguish this issue from knowledge of the forward measurement-error mechanism. Since \(P_U\) is known, for any fixed \(z\in\mathcal B\) and suitable \(L\), $(\mathcal A_U L)(z)
=
\int_{\mathcal B}L(z+u)\,P_U(du)$
is, in principle, determined by the measurement-error distribution. The computational difficulty arises instead from the \(L^2(P_X)\)-based inverse formulation. In particular,

$$
\|K_{h,x}\|_{L^2(P_X)}^2
=
\int_{\mathcal B}K_{h,x}^2(z)\,P_X(dz)
$$

depends on the unknown law of the unobserved predictor \(X\), while the adjoint 
$A^*:L^2(P_X)\to L^2(P_W)$ is defined through
$\langle AL,g\rangle_{L^2(P_X)}
=
\langle L,A^*g\rangle_{L^2(P_W)}.$
Equivalently, for suitable \(g\),
$(A^*g)(W)
=
\mathbb{E}\{g(X)\mid W\},$ which involves the latent conditional distribution of \(X\) given \(W\). Hence, the population Tikhonov inverse is not directly computable from \(\{(Y_i,W_i)\}_{i=1}^n\).

One possible route toward a feasible construction is to avoid estimating \(P_X\) explicitly and instead approximate the inverse weight within a finite-dimensional sieve space. For example, let
$
\mathcal L_J= \operatorname{span}\{\psi_1,\ldots,\psi_J\}
\subset L^2(P_W)
$ and represent
$
L_J(w)
=
\sum_{j=1}^J\theta_j\psi_j(w).
$

The population inverse problem could then be replaced by a regularized empirical problem for the coefficient vector \(\theta=(\theta_1,\ldots,\theta_J)^\top\), provided that the criterion can be expressed entirely in terms of observable quantities and the known error law \(P_U\). Such sieve or estimated-operator regularization is a standard strategy in statistical inverse problems. However, in the present measurement-error setting, a naive empirical replacement of $A:L^2(P_W)\to L^2(P_X)$
is not sufficient, because evaluating \(AL\) at sample points would require the unobserved \(X_i\). A feasible criterion must therefore be constructed so that the latent predictors do not enter directly.

Accordingly, the Tikhonov weights considered in the present analysis should be interpreted as population-level, or oracle, inverse weights. Their purpose is to establish existence, approximation, and stability properties and to provide sufficient conditions for the asymptotic theory. Developing a fully data-driven sieve or estimated-operator analogue that depends only on the contaminated observations and the known measurement-error distribution is an important extension, but is beyond the scope of the present work.
\end{remark}

\subsection{Partial contamination model}
\label{sec:partial}

The general inverse problem need not admit an explicit or stable solution. We first consider an important special case in which the measurement-error operator can be inverted directly, without requiring knowledge of the latent distribution \(P_X\). This provides both a concrete construction of the inverse weight and a useful benchmark for the more general regularized formulation considered subsequently. Suppose that
\begin{equation}\label{eq:partial}
P_U = p_0\delta_0+(1-p_0)\nu, \qquad \frac{1}{2}<p_0\le 1,
\end{equation}
where \(\delta_0\) denotes the point mass at the origin of \(\mathcal B\) and \(\nu\) is a Borel probability measure on \(\mathcal B\). Thus, with probability \(p_0\), the predictor is observed without contamination, whereas with probability \(1-p_0\), the measurement error is drawn from \(\nu\). Models of this form have been considered previously; 
see \citet{Hesse1995,YuanChen2002} and, in connection with well-posedness, \citet{AnHu2012}.

Define
\begin{equation}\label{eq:Qop}
(\mathcal{Q}f)(z) =
\int_{\mathcal{B}} f(z+u)\,\nu(du),
\qquad
q=\frac{1-p_0}{p_0}<1.
\end{equation}
Consider $\mathcal{Q}$ as a linear operator on the Banach space $\mathcal{B}_b(\mathcal{B})$ of bounded Borel functions equipped with the supremum norm $
\|f\|_\infty
=
\sup_{z\in\mathcal{B}} |f(z)|.
$
For every $f\in\mathcal{B}_b(\mathcal{B})$,
$
|(\mathcal{Q}f)(z)|
\le
\int_{\mathcal{B}} |f(z+u)|\,\nu(du)
\le
\|f\|_\infty,
$
and therefore
$
\|\mathcal{Q}f\|_\infty
\le
\|f\|_\infty.
$
Hence,
$
\|\mathcal{Q}\|\le 1.
$
Under the partial contamination model
\eqref{eq:partial}, the measurement-error operator satisfies
\begin{equation}
(\mathcal{A}_U f)(z)
=
p_0 f(z)
+
(1-p_0)\int_{\mathcal{B}} f(z+u)\,\nu(du),    
\end{equation}
so that
$$
\mathcal{A}_U
=
p_0 I+(1-p_0)\mathcal{Q}
=
p_0(I+q\mathcal{Q}).
$$
Since
$
q\|\mathcal{Q}\|
\le
q
=
\frac{1-p_0}{p_0}
<1,
$
the operator $I+q\mathcal{Q}$ is invertible on $\mathcal{B}_b(\mathcal{B})$. By the classical Neumann-series theorem,
$$
(I+q\mathcal{Q})^{-1}
=
\sum_{j=0}^{\infty}(-q)^j\mathcal{Q}^j;
$$
see \citet[Theorem~2.14]{Kress2014}. Consequently,
$$
\mathcal{A}_U^{-1}
=
\frac{1}{p_0}
\sum_{j=0}^{\infty}(-q)^j\mathcal{Q}^j.
$$
The above series converges in operator norm because
$$
\sum_{j=0}^{\infty}
\|(-q)^j\mathcal{Q}^j\|
\le
\sum_{j=0}^{\infty}
q^j\|\mathcal{Q}\|^j
\le
\sum_{j=0}^{\infty} q^j
=
\frac{1}{1-q}
<\infty.
$$
Therefore, for every bounded Borel function
$K:\mathcal{B}\to\mathbb{R}$, the inverse weight
$$
L
=
\mathcal{A}_U^{-1}K
=
\frac{1}{p_0}
\sum_{j=0}^{\infty}
(-q)^j\mathcal{Q}^jK
$$
is well defined, and the series converges uniformly on $\mathcal{B}$. By construction, $\mathcal{A}_U L=K.$
The same argument also yields a stability bound for the inverse operator. Indeed,

$$
\|\mathcal{A}_U^{-1}\|
\le
\frac{1}{p_0}
\sum_{j=0}^{\infty}
q^j\|\mathcal{Q}^j\|
\le
\frac{1}{p_0}
\sum_{j=0}^{\infty}q^j
=
\frac{1}{p_0(1-q)} = \frac{1}{2p_0-1}.
$$
Hence,
$
\|L\|_\infty
=
\|\mathcal{A}_U^{-1}K\|_\infty
\le
\|\mathcal{A}_U^{-1}\|\,\|K\|_\infty
\le
\frac{\|K\|_\infty}{2p_0-1}.
$
Thus, the condition $p_0>1/2$ guarantees both the existence of the Neumann-series inverse and a quantitative stability bound for the resulting inverse weight.
The bound
$
\|\mathcal A_U^{-1}\|
\le
\frac{1}{2p_0-1}
$
also provides a direct interpretation of the stability of the inverse problem under partial contamination. When \(p_0\) is close to one, most observations are uncontaminated, \(q=(1-p_0)/p_0\) is small, and the inverse operator remains well controlled. As \(p_0\downarrow1/2\), however, the bound increases and eventually diverges, indicating that inversion becomes progressively more sensitive to perturbations as the contaminated component becomes more prominent. Thus, the condition \(p_0>1/2\) is not merely technical: it ensures convergence of the Neumann series and provides a quantitative stability guarantee for the resulting inverse weight.

\begin{remark}
The preceding analysis treats the inverse-weight construction in a general operator framework and, in particular, allows for situations in which exact inversion is unavailable or unstable. We now consider an important special case in which the measurement-error operator admits an explicit bounded inverse. Under this partial contamination model, the inverse weight can be constructed directly from the known measurement-error distribution \(P_U\), without requiring knowledge of the latent distribution \(P_X\). Thus, this setting provides a concrete example in which the inverse-weight construction is both exact and directly implementable.
\end{remark}

\section{Discussion}
\label{sec:disc}

The paper develops a nonparametric regression framework for a latent Banach-valued predictor observed with measurement error. The central idea is to recover the local weighting mechanism required for kernel regression rather than reconstruct the latent predictor itself. This converts the errors-in-variables problem into an operator inversion problem and provides a common framework for estimation and inference with contaminated infinite-dimensional predictors.

The use of a general Banach space is substantive. Kernel regression requires a notion of proximity, but not an inner product. Although many functional observations can be embedded in \(L^2\), the resulting squared-integrable geometry need not be the most appropriate notion of similarity for a given scientific problem. Alternatives such as the \(L^1\) norm or the supremum norm on \(C(\mathcal T)\) emphasize, respectively, accumulated absolute discrepancy or the largest localized departure. Both lead to non-Hilbert Banach geometries. This flexibility is especially relevant for high-resolution functional data, where clinically or scientifically meaningful differences may be localized, distributional, or driven by extremes rather than by average squared separation. Similar general Banach and semimetric formulations have long been used in functional nonparametric regression \citep{ferraty2006nonparametric,ferratymasvieu2007,meilan2024nonparametric}.

This distinction becomes more important when the predictor is contaminated. Measurement error changes the distances from a target profile and therefore changes the local neighborhood on which the kernel estimator is based; it is not merely an additional source of variance. This issue arises naturally in wearable accelerometry, continuous glucose monitoring, spectroscopy, hyperspectral imaging, and other sensor-derived functional data. For example, accelerometry studies have shown that clinically relevant information can be concentrated in upper-tail and time-localized activity behavior \citep{niyogi2026quantifying}, while CGM trajectories are affected by calibration error, sensor noise, physiological lag, and sensor degradation \citep{bequette2010continuous,lodwig2003continuous}. Contaminated functional predictors have also been studied directly in nonparametric and functional regression settings \citep{ferraty2019nonparametric,jadhav2020functional}. The present approach differs by correcting the weighting functional used by nonparametric regression rather than imposing a linear model or first reconstructing the entire latent trajectory.

The inverse-problem formulation also clarifies an important distinction between general and structured contamination. For a general error mechanism, stable recovery may require regularization because exact inversion can be unavailable or ill posed. Within this framework, the proposed methodology develops both pointwise and uniform theory for the regression function, allowing local inference at a fixed predictor profile as well as simultaneous inference over a prescribed region of the predictor space. Under partial contamination, however, the measurement-error operator admits an explicit Neumann-series inverse, yielding an exact inverse weight that depends only on the known error law. The corresponding operator bound gives a direct interpretation of stability and shows how the severity of contamination affects the conditioning of the inverse problem. This contrast illustrates that the difficulty of measurement error is determined not only by its magnitude but also by the structure of the error distribution.

More broadly, the framework suggests treating functional geometry, measurement error, and statistical inversion as parts of a single problem. This perspective is relevant whenever the observed functional predictor is itself the output of a noisy measurement system and the scientifically meaningful notion of similarity need not be Hilbertian. Future work can build on this formulation through data-driven inverse constructions, richer error mechanisms, and adaptive regularization, while retaining the basic principle that the goal is to recover the statistical information needed for regression rather than necessarily reconstruct the full latent signal.

\newpage
\section*{Appendix}

\begin{proof}[Proof of Lemma \ref{lem:technical-details}]

For any integrable function $f(Y,W)$, write
\begin{equation}\label{eq:P-Pn}
\mathbb{P}f=\mathbb{E}\{f(Y,W)\}, \qquad \mathbb{P}_nf=\frac{1}{n}\sum_{i=1}^nf(Y_i,W_i).
\end{equation}

For (i), equation~\eqref{eq:conditional-action} and the definition of $R_{\lambda,h,x}$ imply
\[
\begin{aligned}
D^L_{\lambda,h}(x) &=\mathbb{E}\{(\mathcal{A}_UL_{\lambda,h,x})(X)\} = D_h(x)+r_{0,\lambda,h}(x),\\
N^L_{\lambda,h}(x) &=\mathbb{E}\!\left[m(X)(\mathcal{A}_UL_{\lambda,h,x})(X)\right] = N_h(x)+r_{1,\lambda,h}(x).
\end{aligned}
\]
For completeness, \eqref{eq:conditional-action} follows by taking any bounded Borel function $g:\mathcal{B}\to\mathbb{R}$ and using independence and Fubini's theorem:
\[
\mathbb{E}\{g(X)L(X+U)\} = \int_{\mathcal{B}}g(z)\int_{\mathcal{B}}L(z+u)\,P_U(du)\,P_X(dz) = \mathbb{E}\{g(X)(\mathcal{A}_UL)(X)\}.
\]
Since $U$ is independent of $(X,Y)$, $\mathbb{E}\{L(X+U)\mid X,Y\}=(\mathcal{A}_UL)(X)$; conditioning once more and using $\mathbb{E}(Y\mid X)=m(X)$ proves the second identity in \eqref{eq:conditional-action}.

For (ii),
\[
m_{h_n}(x)-m(x) = \frac{\mathbb{E}[\{m(X)-m(x)\}K_{h_n,x}(X)]}{D_{h_n}(x)}.
\]
For all sufficiently large $n$, $h_n\le\delta$. Because $\operatorname{supp}(K)\subseteq[0,1]$, the event $K_{h_n,x}(X)\ne0$ implies $\|X-x\|_{\mathcal{B}}\le h_n$. Assumption~\ref{cond:m} and $K\ge0$ therefore give
\[
|m_{h_n}(x)-m(x)| \le \frac{L_mh_n^\alpha\mathbb{E}K_{h_n,x}(X)}{D_{h_n}(x)} = L_mh_n^\alpha.
\]

For (iii), set $D_n:=D_{h_n}(x)$, $r_{0n}:=r_{0,\lambda_n,h_n}(x)$, and $s_{0n}:=s_{0,\lambda_n,h_n}(x)$. By (i), $P L_n(W)=D_n+r_{0n}$, while
\[
(P_n-P)L_n(W) = O_{\mathbb{P}}\!\left(\frac{s_{0n}}{\sqrt{n}}\right)
\]
by Chebyshev's inequality. Consequently,
\[
\frac{\widehat{D}_{\lambda_n,h_n}(x)}{D_n} = 1+\frac{r_{0n}}{D_n} + O_{\mathbb{P}}\!\left(\frac{s_{0n}}{\sqrt{n}D_n}\right) \overset{\bbP}{\longrightarrow}1
\]
by Assumptions~\ref{cond:residual} and \ref{cond:L}. The second convergence in \eqref{eq:den-relative} follows directly from $D^L_{\lambda_n,h_n}(x)=D_n+r_{0n}$. Since $D_n>0$ eventually, \eqref{eq:den-positive} follows as well.

It remains to prove (iv). Let
\begin{equation}\label{eq:Zni}
Z_{n,i}=\{Y_i-m^L_{\lambda_n,h_n}(x)\}L_n(W_i), \qquad i=1,\cdots,n.
\end{equation}
Then $\mathbb{E}Z_{n,i}=0$, and the triangle inequality in $L^2(\mathbb{P})$ yields
\begin{equation}\label{eq:tau-bound}
\{\operatorname{Var}(Z_{n,1})\}^{1/2} \le s_{1,\lambda_n,h_n}(x) + |m^L_{\lambda_n,h_n}(x)|s_{0,\lambda_n,h_n}(x).
\end{equation}
Moreover, (i), $|N_{h_n}(x)|\le M_mD_{h_n}(x)$, and Assumption~\ref{cond:residual} show that, for all sufficiently large $n$,
\begin{equation}\label{eq:bound-of-m}
 |m^L_{\lambda_n,h_n}(x)| \le \frac{M_mD_{h_n}(x)+|r_{1,\lambda_n,h_n}(x)|}{D_{h_n}(x)-|r_{0,\lambda_n,h_n}(x)|} \le 2M_m+1.   
\end{equation}
Finally, the exact identity
\begin{equation}\label{eq:exact-sampling}
\widehat{m}_{\lambda_n,h_n}(x)-m^L_{\lambda_n,h_n}(x) = \frac{(P_n-P)[\{Y-m^L_{\lambda_n,h_n}(x)\}L_n(W)]}{\widehat{D}_{\lambda_n,h_n}(x)}
\end{equation}
follows by subtracting $m^L_{\lambda_n,h_n}(x)$ from the ratio in \eqref{eq:estimator}. Applying Chebyshev's inequality to the numerator, using \eqref{eq:tau-bound}, and invoking (iii) proves \eqref{eq:sampling-bound}.
\end{proof}

\begin{proof}[Proof of Theorem \ref{thm:main}]
The error admits the decomposition
\begin{equation}\label{eq:three-term}
\begin{aligned}
\widehat{m}_{\lambda_n,h_n}(x)-m(x)
&=\{\widehat{m}_{\lambda_n,h_n}(x)-m^L_{\lambda_n,h_n}(x)\}\\
&\quad+\{m^L_{\lambda_n,h_n}(x)-m_{h_n}(x)\}
+\{m_{h_n}(x)-m(x)\}.
\end{aligned}
\end{equation}
We bound the three terms in turn. First, the identities in Lemma~\ref{lem:technical-details} give
\[
m^L_{\lambda_n,h_n}(x)-m_{h_n}(x)
=\frac{r_{1,\lambda_n,h_n}(x)-m_{h_n}(x)r_{0,\lambda_n,h_n}(x)}
{D_{h_n}(x)+r_{0,\lambda_n,h_n}(x)}.
\]
Because $K\ge0$ and $|m(X)|\le M_m$, we have $|m_{h_n}(x)|\le M_m$. Assumption~\ref{cond:residual} also implies that the absolute value of the denominator is at least $D_{h_n}(x)/2$ for all sufficiently large $n$. Hence
\begin{equation}\label{eq:inverse-bound-final}
|m^L_{\lambda_n,h_n}(x)-m_{h_n}(x)|
\le 2(1+M_m)
\frac{|r_{0,\lambda_n,h_n}(x)|+|r_{1,\lambda_n,h_n}(x)|}
{D_{h_n}(x)}.
\end{equation}
Second, the bound in Lemma~\ref{lem:technical-details} yields
\begin{equation}
\widehat{m}_{\lambda_n,h_n}(x)-m^L_{\lambda_n,h_n}(x)
=O_{\mathbb{P}}\!\left(
\frac{s_{0,\lambda_n,h_n}(x)+s_{1,\lambda_n,h_n}(x)}
{\sqrt{n}\,D_{h_n}(x)}\right).
\end{equation}
Finally, Lemma~\ref{lem:technical-details} gives
\begin{equation}\label{eq:smoothing-bias}
|m_{h_n}(x)-m(x)|\le L_mh_n^\alpha.
\end{equation}
Combining \eqref{eq:inverse-bound-final}--\eqref{eq:smoothing-bias} in \eqref{eq:three-term} proves \eqref{eq:main-rate}. Each term on the right-hand side of \eqref{eq:main-rate} converges to zero by the assumptions, which proves \eqref{eq:consistency}. Equation~\eqref{eq:den-positive} follows from Lemma~\ref{lem:technical-details}.
\end{proof}

\begin{proof}[Proof of Lemma~\ref{lemma:sigma}]
First, we will show
\begin{equation}
\label{ineq:D}
    \{D_n^L(x)\sigma_n^2(x)\}^{1/2}
    \le
    s_{1,\lambda_n,h_n}(x)
    +
    |m^L_{\lambda_n,h_n}(x)|\,s_{0,\lambda_n,h_n}(x).
\end{equation}

\noindent Recall that $D_n^{L}(x)= D_{\lambda_n,h_n}^{L}(x)$. 
Fix $x\in\mathcal B$ and take $n$ sufficiently large that
$D_n^L(x)>0$. Recall that $L_n=L_{\lambda_n,h_n,x}$ and
\[
D_n^L(x)=\mathbb E\{L_n(W)\},\qquad
m^L_{\lambda_n,h_n}(x)
=\frac{\mathbb E\{YL_n(W)\}}{D_n^L(x)}.
\]
Hence,
\[
\mathbb E\{YL_n(W)\}
=m^L_{\lambda_n,h_n}(x)\mathbb E\{L_n(W)\},
\]
which gives the centered decomposition
\begin{align*}
\{Y-m^L_{\lambda_n,h_n}(x)\}L_n(W)
&=\bigl[YL_n(W)-\mathbb E\{YL_n(W)\}\bigr]\\
&\quad
-m^L_{\lambda_n,h_n}(x)
 \bigl[L_n(W)-\mathbb E\{L_n(W)\}\bigr].
\end{align*}

Write $\|Z\|_{L^2(\mathbb{P})}=(\mathbb E|Z|^2)^{1/2}$. 
Assumption~\ref{cond:L} implies that both centered random variables on the right belong
to $L^2$. By the definition of $\sigma_n^2(x)$ in \eqref{eq:sigma-definition} and the
$L^2$ triangle inequality,
\begin{align*}
\{D_n^L(x)\sigma_n^2(x)\}^{1/2}
&=\bigl\|\{Y-m^L_{\lambda_n,h_n}(x)\}L_n(W)\bigr\|_{L^2(\mathbb{P})}\\
&\le
\bigl\|YL_n(W)-\mathbb E\{YL_n(W)\}\bigr\|_{L^2(\mathbb{P})}\\
&\quad+
|m^L_{\lambda_n,h_n}(x)|
\bigl\|L_n(W)-\mathbb E\{L_n(W)\}\bigr\|_{L^2(\mathbb{P})}\\
&=
\{\operatorname{Var}(YL_n(W))\}^{1/2}
+
|m^L_{\lambda_n,h_n}(x)|
\{\operatorname{Var}(L_n(W))\}^{1/2}\\
&=
s_{1,\lambda_n,h_n}(x)
+
|m^L_{\lambda_n,h_n}(x)|s_{0,\lambda_n,h_n}(x).
\end{align*}
This proves the inequality \eqref{ineq:D}. Also, by Assumptions \ref{cond:D} and \ref{cond:residual}, 
\begin{equation}\label{expressionD_n^L}
  \frac{D_n^{L}(x)}{D_{h_n}(x)}=1 + \frac{r_{0,\lambda_n,h_n}(x)}{D_{h_n}(x)} \to 1 \qquad \text{as}\qquad n \to \infty.  
\end{equation}
Hence, $D_n^{L}(x)>0$ for sufficiently large $n$. Also, by \eqref{eq:bound-of-m}, we have $m^L_{\lambda_n,h_n}(x)=O(1)$.
Hence there exists a constant $C \ge 1$ such that $|m^L_{\lambda_n,h_n}(x)| \le C$ for all sufficiently large $n$. Dividing both sides of Lemma~\ref{lemma:sigma} by $\sqrt{n}D_n^L(x)$, we get 
\begin{align*}
0\le\frac{\sigma_n(x)}{\sqrt{nD_n^L(x)}}
&\le
\frac{s_{1,\lambda_n,h_n}(x)
      +|m^L_{\lambda_n,h_n}(x)|s_{0,\lambda_n,h_n}(x)}
     {\sqrt n\,D_n^L(x)}\\
&\le
C\,\frac{D_{h_n}(x)}{D_n^L(x)}
\frac{s_{0,\lambda_n,h_n}(x)+s_{1,\lambda_n,h_n}(x)}
     {\sqrt n\,D_{h_n}(x)}
\longrightarrow0,
\end{align*}
by \eqref{expressionD_n^L} and Assumption~\ref{cond:L}. Now we square both sides to get
\[
\frac{\sigma_n^2(x)}{nD^L_{\lambda_n,h_n}(x)}
=\frac{\sigma_n^2(x)}{nD_n^L(x)}
\longrightarrow0, \qquad \text{as}\;\; n \to \infty.
\]
\end{proof}
\begin{proof}[Proof of Theorem~\ref{thm:pointwise-asymptotic-normality}]
By the definition of $m_n^L$,
\[
\E\{Q_n(Y,W)\}
=\E\{YL_n(W)\}-m_n^L\E\{L_n(W)\}=0,
\]
and
\[
\operatorname{Var}\{Q_n(Y,W)\}
=\E\{Q_n^2(Y,W)\}
=D_n^L\sigma_n^2.
\]

We first establish $\dfrac{\widehat D_n}{D_n^L}
\overset{\mathbb{P}}{\longrightarrow}1.$ Since
$\E(\widehat D_n)=D_n^L$, Chebyshev's inequality gives, for every
$\eta>0$,
\[
\mathbb{P}\!\left(
\left|\frac{\widehat D_n}{D_n^L}-1\right|>\eta
\right)
\le
\frac{\operatorname{Var}\{L_n(W)\}}
{n\eta^2(D_n^L)^2}
\longrightarrow0,
\]
which is what we wanted. 
\medskip

For $i=1,\ldots,n$, let $Q_{n,i}=\{Y_i-m_n^L\}L_n(W_i)$
and define the triangular array $X_{n,i}
=\dfrac{Q_{n,i}}{\sqrt{nD_n^L}\,\sigma_n}.$
For each $n$, the variables $X_{n,1},\ldots,X_{n,n}$ are independent,
have mean zero, and satisfy
\[
\sum_{i=1}^n\E(X_{n,i}^2)
=\frac{n\E(Q_{n,1}^2)}{nD_n^L\sigma_n^2}
=1.
\]
Moreover, for every $\varepsilon>0$,
\begin{align*}
\sum_{i=1}^n
\E\!\left[
X_{n,i}^2\mathbf 1\{|X_{n,i}|>\varepsilon\}
\right]
&=
\frac{
\E\!\left[
Q_n^2(Y,W)
\mathbf 1\!\left\{
|Q_n(Y,W)|>
\varepsilon\sqrt{nD_n^L}\,\sigma_n
\right\}
\right]
}{D_n^L\sigma_n^2}
\longrightarrow0
\end{align*}
by \eqref{eq:lindeberg-condition}. The Lindeberg--Feller central limit
theorem therefore yields
\begin{equation}
\label{eq:numerator-clt}
\frac{\sum_{i=1}^nQ_{n,i}}
{\sqrt{nD_n^L}\,\sigma_n}
\overset{d}{\longrightarrow}N(0,1).
\end{equation}

The ratio estimator satisfies the exact identity
\begin{align*}
\widehat m_n-m_n^L
&=
\frac{\mathbb{P}_n\{YL_n(W)\}-m_n^L\mathbb{P}_n\{L_n(W)\}}{\widehat D_n}\\
&=
\frac{\mathbb{P}_n Q_n(Y,W)}{\widehat D_n}
=\frac{1}{n\widehat D_n}\sum_{i=1}^nQ_{n,i}.
\end{align*}
It follows that
\begin{equation}
\label{eq:factorization}
\frac{\sqrt{nD_n^L}}{\sigma_n}
\{\widehat m_n-m_n^L\}
=
\frac{D_n^L}{\widehat D_n}
\frac{\sum_{i=1}^nQ_{n,i}}
{\sqrt{nD_n^L}\,\sigma_n}.
\end{equation}
Combining $\dfrac{\widehat D_n}{D_n^L}
\overset{\mathbb{P}}{\longrightarrow}1$ and
\eqref{eq:numerator-clt} with Slutsky's theorem proves
\eqref{eq:centered-clt}.

Finally, decompose
\begin{align*}
\frac{\sqrt{nD_n^L}}{\sigma_n}
\{\widehat m_n-m(x)\}
&=
\frac{\sqrt{nD_n^L}}{\sigma_n}
\{\widehat m_n-m_n^L\}+
\frac{\sqrt{nD_n^L}}{\sigma_n}
\{m_n^L-m(x)\}.
\end{align*}
The first term converges in distribution to $N(0,1)$ by
\eqref{eq:centered-clt}, while the second converges to zero by
\eqref{eq:standardized-bias}. A second application of Slutsky's theorem
proves \eqref{eq:true-target-clt}.
\end{proof}
\begin{proof}[Proof of Theorem \ref{thm:uniform-bias}]
We first establish the uniform stochastic approximation and then combine
it with the deterministic smoothing bound.

\medskip
\noindent Fix a sufficiently large $n$ for which
Condition~\ref{unif:population-denominator} holds.  For every
$x\in\mathcal X$,
\begin{align*}
\mathbb{P} f_{1,n,x}
&=
\E\!\left[
\{Y-m^L_{\lambda_n,h_n}(x)\}L_{n,x}(W)
\right]\\
&=
N_n^L(x)-m^L_{\lambda_n,h_n}(x)D_n^L(x)
=0.
\end{align*}

Thus $f_{1,n,x}=Z_{1,n,x}$ almost surely.  On
$\widehat D_n^L(x)\ne0$, we have:
\begin{align}\label{eq:exact-centered-ratio}
&\widehat m_{\lambda_n,h_n}(x)
-m^L_{\lambda_n,h_n}(x)
\nonumber\\
&\quad=
\frac{
\widehat N_n^L(x)-N_n^L(x)
-m^L_{\lambda_n,h_n}(x)
\{\widehat D_n^L(x)-D_n^L(x)\}
}{\widehat D_n^L(x)}
\nonumber\\
&\quad=
\frac{(\mathbb{P}_n-\mathbb{P})f_{1,n,x}}
{\widehat D_n^L(x)}.
\end{align}

\medskip
\noindent For each $n$, take an $\eta_n$-net $\{x_{n,1},\ldots,x_{n,M_n}\}$ of $\mathcal X$.  By Condition~\ref{unif:bernstein} and Bernstein's
inequality, for every $t>0$, every $1\le k\le M_n$, and
$j\in\{0,1\}$,
\begin{equation}
\mathbb{P}\left\{
\left|(\mathbb{P}_n-\mathbb{P})f_{j,n,x_{n,k}}\right|>t
\right\}
\le
2\exp\!\left\{
-\frac{nt^2}{2(v_{j,n}^2+b_{j,n}t)}
\right\}.
\label{eq:bernstein-one-net-point}
\end{equation}
The union bound therefore yields
\begin{equation}
\mathbb{P}\left\{
\max_{1\le k\le M_n}
\left|(\mathbb{P}_n-\mathbb{P})f_{j,n,x_{n,k}}\right|>t
\right\}
\le
2M_n\exp\!\left\{
-\frac{nt^2}{2(v_{j,n}^2+b_{j,n}t)}
\right\}.
\label{eq:bernstein-net}
\end{equation}

\noindent For $u>0$, put $s_n(u):=\ell_n+u$ and
\[
t_{j,n}(u)
:=
\sqrt{\frac{2v_{j,n}^2s_n(u)}{n}}
+2b_{j,n}\frac{s_n(u)}{n}.
\]
If $r=s_n(u)/n$, direct expansion gives
\[
t_{j,n}^2(u)
-2r\{v_{j,n}^2+b_{j,n}t_{j,n}(u)\}
\ge0.
\]
Consequently, the exponent in \eqref{eq:bernstein-net}, evaluated at
$t=t_{j,n}(u)$, is at least $s_n(u)$.  Since
$\ell_n=\log(2M_n)$,
\begin{equation}\label{eq:net-high-probability}
\mathbb{P}\left\{
\max_{1\le k\le M_n}
\left|(\mathbb{P}_n-\mathbb{P})f_{j,n,x_{n,k}}\right|
>t_{j,n}(u)
\right\}
\le e^{-u}.
\end{equation}
Because $\ell_n\ge\log2$, for each fixed $u$ the threshold in
\eqref{eq:net-high-probability} is bounded by a constant depending only
on $u$ times
\[
v_{j,n}\sqrt{\frac{\ell_n}{n}}
+b_{j,n}\frac{\ell_n}{n}.
\]
It follows that
\begin{equation}
\max_{1\le k\le M_n}
\left|(\mathbb{P}_n-\mathbb{P})f_{j,n,x_{n,k}}\right|
=
O_{\mathbb{P}}\!\left(
v_{j,n}\sqrt{\frac{\ell_n}{n}}
+b_{j,n}\frac{\ell_n}{n}
\right).
\label{eq:net-stochastic-order}
\end{equation}

\medskip
\noindent For every $x\in\mathcal X$, choose a net point $\pi_n(x)$ satisfying
$\|x-\pi_n(x)\|_{\mathcal B}\le\eta_n$.  The triangle inequality gives
\begin{align*}
\sup_{x\in\mathcal X}
\left|(\mathbb{P}_n-\mathbb{P})f_{j,n,x}\right|
&\le
\max_{1\le k\le M_n}
\left|(\mathbb{P}_n-\mathbb{P})f_{j,n,x_{n,k}}\right|\\
&\quad+
\sup_{x\in\mathcal X}
\left|(\mathbb{P}_n-\mathbb{P})
\{f_{j,n,x}-f_{j,n,\pi_n(x)}\}\right|.
\end{align*}
By Condition~\ref{unif:modulus},
\begin{align*}
&\sup_{x\in\mathcal X}
\left|(\mathbb{P}_n-\mathbb{P})
\{f_{j,n,x}-f_{j,n,\pi_n(x)}\}\right|\\
&\quad\le
\eta_n^{\gamma_j}
\{\mathbb{P}_n A_{j,n}+\mathbb{P} A_{j,n}\}.
\end{align*}
Moreover, $A_{j,n}\ge0$ and
$\mathbb{P} A_{j,n}\le c_{j,n}$.  Markov's inequality gives, for every
$M>0$,
\[
\mathbb{P}\{\mathbb{P}_n A_{j,n}>Mc_{j,n}\}
\le\frac1M.
\]
Hence $\mathbb{P}_n A_{j,n}=O_{\mathbb{P}}(c_{j,n})$.  Combining this with
\eqref{eq:net-stochastic-order} proves
\begin{equation}
\Delta_{j,n}
:=
\sup_{x\in\mathcal X}
\left|(\mathbb{P}_n-\mathbb{P})f_{j,n,x}\right|
=O_{\mathbb{P}}(a_{jn}),
\qquad j\in\{0,1\}.
\label{eq:uniform-empirical-process-rate}
\end{equation}

\medskip
\noindent For $j=0$, the quantity in
\eqref{eq:uniform-empirical-process-rate} is exactly
$\Delta_{0,n}
=
\sup_{x\in\mathcal X}
|\widehat D_n^L(x)-D_n^L(x)|.$
By Conditions~\ref{unif:population-denominator} and
\ref{unif:tuning},
\[
\frac{\Delta_{0,n}}{d_n}
=
O_{\mathbb{P}}\!\left(\frac{a_{0n}}{d_n}\right)
=o_{\mathbb{P}}(1).
\]
Let
\[
\mathcal E_n
:=
\left\{\Delta_{0,n}\le\frac{d_n}{2}\right\}.
\]
Then $\mathbb{P}(\mathcal E_n)\to1$.  On $\mathcal E_n$,
\begin{align*}
\inf_{x\in\mathcal X}\widehat D_n^L(x)
&\ge
\inf_{x\in\mathcal X}D_n^L(x)-\Delta_{0,n}\\
&\ge d_n-\frac{d_n}{2}
=\frac{d_n}{2}>0.
\end{align*}
This proves \eqref{eq:uniform-empirical-denominator-positive}.

\medskip
\noindent On $\mathcal E_n$, identity \eqref{eq:exact-centered-ratio} is valid
simultaneously for all $x\in\mathcal X$.  Therefore,
\begin{align}
S_n
&:=
\sup_{x\in\mathcal X}
\left|
\widehat m_{\lambda_n,h_n}(x)
-m^L_{\lambda_n,h_n}(x)
\right|
\nonumber\\
&\le
\frac{2}{d_n}
\sup_{x\in\mathcal X}
\left|(\mathbb{P}_n-\mathbb{P})f_{1,n,x}\right|
=\frac{2\Delta_{1,n}}{d_n}.
\label{eq:stochastic-ratio-on-good-event}
\end{align}
It follows from \eqref{eq:uniform-empirical-process-rate} that
\begin{equation}
S_n
=O_{\mathbb{P}}\!\left(\frac{a_{1n}}{d_n}\right).
\label{eq:uniform-stochastic-ratio-rate}
\end{equation}
To justify the unconditional order explicitly, note that for every
$M>0$,
\[
\mathbb{P}\left\{
S_n>M\frac{a_{1n}}{d_n}
\right\}
\le
\mathbb{P}(\mathcal E_n^c)
+\mathbb{P}\left\{2\Delta_{1,n}>Ma_{1n}\right\}.
\]
The first term tends to zero, while the second can be made arbitrarily
small uniformly for all sufficiently large $n$ by choosing $M$ large.

\medskip
\noindent Fix a sufficiently large $n$ for which $h_n\le\delta$ and
$D_{h_n}(x)>0$ for every $x\in\mathcal X$.  For each
$x\in\mathcal X$,
\begin{align}
m_{h_n}(x)-m(x)
&=
\frac{
\E\!\left[
\{m(X)-m(x)\}K_{h_n,x}(X)
\right]
}{D_{h_n}(x)}.
\label{eq:smoothing-bias-identity}
\end{align}
Because $\operatorname{supp}(K)\subseteq[0,1]$, the condition
$K_{h_n,x}(X)\ne0$ implies
\[
\|X-x\|_{\mathcal B}\le h_n\le\delta.
\]
On this event, the uniform H\"older condition gives
\[
|m(X)-m(x)|
\le
L_m\|X-x\|_{\mathcal B}^{\alpha}
\le L_mh_n^\alpha.
\]
Since $K\ge0$, equation \eqref{eq:smoothing-bias-identity} therefore
implies
\begin{align*}
|m_{h_n}(x)-m(x)|
&\le
\frac{
\E\!\left[
|m(X)-m(x)|K_{h_n,x}(X)
\right]
}{D_{h_n}(x)}\\
&\le
L_mh_n^\alpha
\frac{\E\{K_{h_n,x}(X)\}}{D_{h_n}(x)}
=L_mh_n^\alpha.
\end{align*}
The constants do not depend on $x$, so
\begin{equation}
\sup_{x\in\mathcal X}|m_{h_n}(x)-m(x)|
\le L_mh_n^\alpha.
\label{eq:uniform-smoothing-bias}
\end{equation}

For every $x\in\mathcal X$, the triangle inequality now gives
\[
|m^L_{\lambda_n,h_n}(x)-m(x)|
\le
|m^L_{\lambda_n,h_n}(x)-m_{h_n}(x)|
+|m_{h_n}(x)-m(x)|.
\]
Taking the supremum over $x$ and using the definition of $\rho_n$ and
\eqref{eq:uniform-smoothing-bias}, we obtain
\begin{align*}
\sup_{x\in\mathcal X}
|m^L_{\lambda_n,h_n}(x)-m(x)|
&\le
\sup_{x\in\mathcal X}
|m^L_{\lambda_n,h_n}(x)-m_{h_n}(x)|\\
&\quad+
\sup_{x\in\mathcal X}|m_{h_n}(x)-m(x)|\\
&\le \rho_n+L_mh_n^\alpha,
\end{align*}
which proves \eqref{eq:uniform-population-bias}.

\medskip
\noindent The triangle inequality and \eqref{eq:uniform-population-bias} imply
\begin{align}
T_n
&:=
\sup_{x\in\mathcal X}
|\widehat m_{\lambda_n,h_n}(x)-m(x)|
\nonumber\\
&\le
S_n+
\sup_{x\in\mathcal X}
|m^L_{\lambda_n,h_n}(x)-m(x)|
\nonumber\\
&\le S_n+\rho_n+L_mh_n^\alpha.
\label{eq:final-triangle-bound}
\end{align}
Let
\[
q_n
:=
\rho_n+h_n^\alpha+\frac{a_{1n}}{d_n},
\qquad
C_m:=\max\{1,L_m\}.
\]
By \eqref{eq:uniform-stochastic-ratio-rate}, for every
$\varepsilon>0$ there are $M_\varepsilon<\infty$ and
$n_\varepsilon$ such that
\[
\mathbb{P}\left\{
S_n>M_\varepsilon\frac{a_{1n}}{d_n}
\right\}
\le\varepsilon,
\qquad n\ge n_\varepsilon.
\]
On the complementary event, \eqref{eq:final-triangle-bound} yields
\[
T_n
\le
M_\varepsilon\frac{a_{1n}}{d_n}
+\rho_n+L_mh_n^\alpha
\le
(M_\varepsilon+C_m)q_n.
\]
This is precisely $T_n=O_{\mathbb{P}}(q_n)$, proving
\eqref{eq:uniform-total-rate}.  If $q_n\to0$, the same display gives
$T_n\overset{\mathbb P}{\longrightarrow}0$, which is uniform consistency.
\end{proof}

\begin{proof}[Proof of Proposition~\ref{prop:dense-range}]
To prove this, we first note that the operator $A$ is a contraction and its adjoint operator is given by $(A^*g)(W)=\mathbb{E}\{g(X) \mid W\}.$ It is sufficient to prove that $\operatorname{ker}(A^*)=\{0\}.$ Let $g \in L^2(P_X)$ satisfy $A^*g=0$. Then for every $\ell \in \mathcal{B}^*$, using the decomposition $W=X+U$ we get
\[0=(A^*g)(W)=\mathbb{E}\!\left[(A^*g)(W)e^{i\ell(W)}\right]=\mathbb{E}\left[g(X)e^{i\ell(W)}\right]=\mathbb{E}\left[g(X)e^{i\ell(X)}\right]\phi_{U}(\ell).\]
Note that since $\mathbb{E}\!\left[g(X)e^{i\ell(W)}\right]= \mathbb{E}\!\left[g(X)e^{i\ell(X)}e^{i\ell(U)}\right],$ and $U$ and $X$ are independent, so $g(X)e^{i\ell(X)}$ and $e^{i\ell(U)}$ are also independent. Therefore, $\mathbb{E}\left[g(X)e^{i\ell(X)}\right]\phi_{U}(\ell)=0$. Since we assume that $\phi_U(\ell) \ne 0$ for every $\ell \in \mathcal{B}^*$, hence $\mathbb{E}\left[g(X)e^{i\ell(X)}\right]=0$ for every $\ell \in \mathcal{B}^*$. Define the finite signed measure $\mu_g(C)=\int_C g\,d P_X.$ Then, its characteristic functional is
\[
\widehat{\mu}_g(\ell) = \int_{\mathcal{B}}e^{i\ell(z)}\,\mu_g(dz) = \mathbb{E}\!\left[g(X)e^{i\ell(X)}\right] = 0, \qquad \ell\in\mathcal{B}^*.
\]
By uniqueness of characteristic functionals of finite signed Radon measures on a separable Banach space, $\mu_g=0$. Hence $\int_C g\,d P_X=0$ for every Borel set $C$, which implies $g=0$ $P_X$-almost surely. Use the identity $\overline{\operatorname{Ran}(A)}=\operatorname{ker}(A^*)^{\perp}$ to finish the proof. 
\end{proof}

\begin{proof}[Proof of Corollary \ref{cor:tikhonov}]
Let $T=AA^*$. Under the source condition $K_{h_n,x}=T^\nu g_{h_n,x}$,
\[
R_{\lambda_n,h_n,x} = -\lambda_n(T+\lambda_n I)^{-1} T^\nu g_{h_n,x}.
\]
By the spectral theorem,
\[
\|R_{\lambda_n,h_n,x}\|_{L^2(P_X)} \le C_\nu\lambda_n^\nu \|g_{h_n,x}\|_{L^2(P_X)},
\]
where
\[
C_\nu = \sup_{t\ge0}\frac{t^\nu}{1+t} <\infty.
\]
Since $|m(X)|\le M_m$ almost surely,
\[
|r_{0,\lambda_n,h_n}(x)| + |r_{1,\lambda_n,h_n}(x)| \le (1+M_m)C_\nu\lambda_n^\nu \|g_{h_n,x}\|_{L^2(P_X)}.
\]
The first rate condition therefore implies Assumption~\ref{cond:residual}.

Next,
\[
L_{\lambda_n,h_n,x} = A^*(T+\lambda_n I)^{-1}K_{h_n,x}.
\]
Again by the spectral theorem,
\[
\|L_{\lambda_n,h_n,x}\|_{L^2(P_W)} \le \frac{1}{2\sqrt{\lambda_n}} \|K_{h_n,x}\|_{L^2(P_X)}.
\]
Because $K\ge0$ and is bounded,
\[
\|K_{h_n,x}\|_{L^2(P_X)}^2 = \mathbb{E} K_{h_n,x}^2(X) \le \|K\|_\infty \mathbb{E} K_{h_n,x}(X) = \|K\|_\infty D_{h_n}(x).
\]
Therefore,
\[
\|L_{\lambda_n,h_n,x}\|_{L^2(P_W)} \le \frac{\sqrt{\|K\|_\infty D_{h_n}(x)}}{2\sqrt{\lambda_n}}.
\]
Under the conditional second-moment assumption on $Y$,
\[
s_{0,\lambda_n,h_n}(x) + s_{1,\lambda_n,h_n}(x) \le C \|L_{\lambda_n,h_n,x}\|_{L^2(P_W)}
\]
for a constant $C<\infty$. Hence, with $C':=C\sqrt{\|K\|_\infty}/2$,
\[
\frac{s_{0,\lambda_n,h_n}(x) + s_{1,\lambda_n,h_n}(x)}{\sqrt{n}\,D_{h_n}(x)} \le \frac{C'}{\sqrt{n\lambda_nD_{h_n}(x)}} \longrightarrow0.
\]
Thus Assumption~\ref{cond:L} also holds, and the conclusion follows from Theorem~\ref{thm:main}.
\end{proof}

\begin{proof}[Proof of Theorem \ref{thm:feasible-uniform-confidence-band}]
We divide the proof into steps.

\medskip
\noindent Let $\Delta_{0,n} := \sup_{x\in\mathcal X} |\widehat D_n^L(x)-D_n^L(x)|.$
The proof of Theorem~\ref{thm:uniform-bias} gives $\Delta_{0,n}=O_{\mathbb P}(a_{0n}).$
Since $d_n=\inf_{x\in\mathcal X}D_n^L(x)$ and Condition~\ref{unif:big} gives $a_{0n}/d_n\to0$, and hence
\[
\sup_{x\in\mathcal X}
\left|
\frac{\widehat D_n^L(x)}{D_n^L(x)}-1
\right|
\le
\frac{\Delta_{0,n}}{d_n}
=o_{\mathbb P}(1).
\]
It follows that
$\mathbb P \left\{\inf_{x\in\mathcal X} \widehat D_n^L(x)\ge\frac{d_n}{2}\right\}
\longrightarrow1.$
Condition~\ref{unif:big} and
$\inf_{x\in\mathcal X}s_n(x)>0$ imply
$\mathbb P \left\{ \inf_{x\in\mathcal X}\widehat s_n(x)>0
\right\} \longrightarrow1.$

\medskip

\noindent Define
\begin{equation}
\widehat{\mathbb Z}_n(x)
:=
\frac{\sqrt n\,\widehat D_n^L(x)}
{\widehat s_n(x)}
\{\widehat m_n(x)-m(x)\}.
\label{eq:feasible-studentized-sampling-process}
\end{equation}
Since
$\mathbb Pf_{n,x}=0$, we have
\begin{equation*}
\widehat m_n(x)-m_n^L(x)
=
\frac{(\mathbb P_n-\mathbb P)f_{n,x}}
{\widehat D_n^L(x)}, \quad \text{whenever}\; \widehat D_n^L(x)\neq0.
\label{eq:exact-uniform-ratio-identity}
\end{equation*}
 Therefore,
\begin{align}
\frac{\sqrt n\,\widehat D_n^L(x)}
{\widehat s_n(x)}
\{\widehat m_n(x)-m_n^L(x)\}
&=
\frac{\mathbb G_nf_{n,x}}{\widehat s_n(x)}
\notag\\
&=
\frac{s_n(x)}{\widehat s_n(x)}
\mathbb G_n\psi_{n,x}.
\label{eq:exact-feasible-studentized-identity}
\end{align}

\noindent Condition \ref{unif:gaussian-coupling} and applying Markov's inequality, we get 
$\left\|
\mathbb G_n\psi_{n,\cdot}
\right\|_{\mathcal X}
=
O_{\mathbb P}(\mathfrak m_n).$
Condition~\ref{unif:big}\textup{(a)} implies
$\sup_{x\in\mathcal X}
\left|
\frac{s_n(x)}{\widehat s_n(x)}-1
\right|
=
o_{\mathbb P}(\mathfrak m_n^{-2}).$ and thus,
\begin{equation}
\left\|
\frac{\sqrt n\,\widehat D_n^L(\cdot)}
{\widehat s_n(\cdot)}
\{\widehat m_n(\cdot)-m_n^L(\cdot)\}
-
\mathbb G_n\psi_{n,\cdot}
\right\|_{\mathcal X}
=
o_{\mathbb P}(\mathfrak m_n^{-1}).
\label{eq:feasible-studentization-remainder}
\end{equation}

\noindent Next, we define the standardized population bias
\[
b_n(x)
:=
\frac{\sqrt n\,D_n^L(x)}{s_n(x)}
\{m_n^L(x)-m(x)\}.
\]
The bias component of \eqref{eq:feasible-studentized-sampling-process}
equals
\[
\frac{\sqrt n\,\widehat D_n^L(x)}
{\widehat s_n(x)}
\{m_n^L(x)-m(x)\}
=
\frac{\widehat D_n^L(x)}{D_n^L(x)}
\frac{s_n(x)}{\widehat s_n(x)}
b_n(x).
\]
By Step 1, Condition~\ref{unif:big}\textup{(a)}, and Condition
\ref{unif:undersmoothing},
\[
\sup_{x\in\mathcal X}
\left|
\frac{\sqrt n\,\widehat D_n^L(x)}
{\widehat s_n(x)}
\{m_n^L(x)-m(x)\}
\right|
=
o_{\mathbb P}(\mathfrak m_n^{-1}).
\]
Combining this result with
\eqref{eq:feasible-studentization-remainder} and
\ref{eq:uniform-Gaussian-coupling} gives
\begin{equation}\label{eq:feasible-Gaussian-coupling}
\left|
\left\|
\widehat{\mathbb Z}_n
\right\|_{\mathcal X}
-
\left\|
\mathbb Z_n^G
\right\|_{\mathcal X}
\right|
=
o_{\mathbb P}(\mathfrak m_n^{-1}).
\end{equation}

\noindent For every $x\in\mathcal X$, $\operatorname{Var}\{\mathbb Z_n^G(x)\}
=
\mathbb P\psi_{n,x}^2
=1,$
and hence
\begin{equation}
\sup_{t\in\mathbb R}
\mathbb P
\left\{
\left|
\|\mathbb Z_n^G\|_{\mathcal X}-t
\right|
\le\varepsilon
\right\}
\le
4\varepsilon\mathfrak m_n,
\qquad \varepsilon>0.
\label{eq:uniform-Gaussian-anti-concentration}
\end{equation}
Combining \eqref{eq:feasible-Gaussian-coupling} and
\eqref{eq:uniform-Gaussian-anti-concentration} yields
\begin{equation}
\sup_{t\in\mathbb R}
\left|
\mathbb P
\left\{
\|\widehat{\mathbb Z}_n\|_{\mathcal X}\le t
\right\}
-
\mathbb P
\left\{
\|\mathbb Z_n^G\|_{\mathcal X}\le t
\right\}
\right|
\longrightarrow0.
\label{eq:feasible-sampling-Gaussian-approximation}
\end{equation}

\medskip
\noindent On the other hand, condition~\ref{unif:big}\textup{(b)} gives
\[
\sup_{t\in\mathbb R}
\left|
\mathbb P_\xi
\left\{
\|\widehat{\mathbb Z}_n^\xi\|_{\mathcal X}\le t
\,\middle|\,
\mathcal D_n
\right\}
-
\mathbb P
\left\{
\|\mathbb Z_n^G\|_{\mathcal X}\le t
\right\}
\right|
\overset{\mathbb P}{\longrightarrow}0.
\]
Thus the conditional multiplier distribution and the sampling
distribution in \eqref{eq:feasible-sampling-Gaussian-approximation}
both approximate the same Gaussian distribution.

\medskip
\noindent To finish the proof, let
$F_n^G(t)
:=
\mathbb P
\left\{
\|\mathbb Z_n^G\|_{\mathcal X}\le t
\right\}.$ The equation
\eqref{eq:uniform-Gaussian-anti-concentration} implies that $F_n^G$ is
continuous. Hence, defining
\[
c_n^G(p)
:=
\inf\{t:F_n^G(t)\ge p\},
\qquad p\in(0,1),
\]
we have
\[
F_n^G\{c_n^G(p)\}=p.
\]

\noindent For every fixed $\eta\in(0,\tau\wedge(1-\tau))$, Condition~\ref{unif:big}\textup{(b)} implies, with probability
approaching one,
\[
c_n^G(1-\tau-\eta)
\le
\widehat c_n(1-\tau)
\le
c_n^G(1-\tau+\eta).
\]
Using \eqref{eq:feasible-sampling-Gaussian-approximation}, it follows
that
\[
1-\tau-\eta+o(1)
\le
\mathbb P
\left\{
\|\widehat{\mathbb Z}_n\|_{\mathcal X}
\le
\widehat c_n(1-\tau)
\right\}
\le
1-\tau+\eta+o(1).
\]
Letting $\eta\downarrow0$ gives
\begin{equation}\label{eq:bootstrap-quantile-coverage}
\mathbb P
\left\{
\|\widehat{\mathbb Z}_n\|_{\mathcal X}
\le
\widehat c_n(1-\tau)
\right\}
=
1-\tau+o(1).
\end{equation}

\noindent On the event where the feasible normalizers are positive,
$\left\{
\|\widehat{\mathbb Z}_n\|_{\mathcal X}
\le
\widehat c_n(1-\tau)
\right\}$
is exactly the event
\[
\left\{
m(x)\in
\widehat{\mathcal C}_{n,1-\tau}(x)
\text{ for every }x\in\mathcal X
\right\}.
\]
The complement of this positivity event has probability tending to
zero by Step 1. Equation \eqref{eq:feasible-uniform-coverage} therefore
follows from \eqref{eq:bootstrap-quantile-coverage}.
\end{proof}

\newpage
\bibliographystyle{unsrtnat} 
\bibliography{main}

\end{document}